\documentclass[11pt,reqno]{amsart}
\usepackage[T1]{fontenc}
\usepackage{amssymb,amsmath}
\usepackage[english]{babel}
\usepackage{enumerate}
\usepackage{enumitem,kantlipsum}
\usepackage[colorlinks=true,linkcolor=blue]{hyperref}
\usepackage{dsfont}
\usepackage{cite}
\usepackage[foot]{amsaddr}
\usepackage{nicefrac}
\usepackage{tikz}
\usetikzlibrary{positioning}

\newtheorem{theorem}{Theorem}[section]
\newtheorem{lemma}[theorem]{Lemma}
\newtheorem{corollary}[theorem]{Corollary}
\theoremstyle{definition}
\newtheorem{definition}[theorem]{Definition}

\newtheorem{remark}[theorem]{Remark}
\newtheorem{assumption}[theorem]{Assumption}

\setlist[enumerate]{leftmargin=*, label=(\roman*)}
\numberwithin{equation}{section}

\def\Ck{{\rm C}_\kappa}
\def\Cb{{\rm C}_{\rm b}}
\def\Cbi{{\rm C}_{\rm b}^\infty}
\def\Lipb{{\rm Lip}_{\rm b}}
\def\Ub{{\rm U}_{\rm b}}
\def\Rb{[-\infty,\infty)}

\DeclareMathOperator{\id}{id}

\DeclareMathOperator{\supp}{supp}
\DeclareMathOperator{\tr}{tr}

\def\B{\mathcal{B}}
\def\d{{\rm d}}
\def\e{\mathbb{E}}
\def\E{\mathcal{E}}
\def\F{\mathcal{F}}
\def\H{\mathcal{H}}
\def\L{\mathcal{L}}
\def\n{\mathcal{N}}
\def\N{\mathbb{N}}
\def\p{\mathcal{P}}
\def\P{\mathbb{P}}
\def\R{\mathbb{R}}
\def\S{\mathbb{S}}
\def\T{\mathcal{T}}
\def\W{\mathcal{W}}
\def\epsilon{\varepsilon}
\def\phi{\varphi}
\def\one{\mathds{1}}

\begin{document}

\title[Limit theorems]{Nonlinear semigroups and limit theorems for convex expectations}

\author{Jonas Blessing$^{*,1}$}
\address{$^*$Department of Mathematics, ETH Zurich, 8092 Zurich, Switzerland}
\email{$^1$jonas.blessing@math.ethz.ch}

\author{Michael Kupper$^{**,2}$}
\address{$^{**}$Department of Mathematics and Statistics,  University of Konstanz, 78457 Konstanz, Germany}
\email{$^2$kupper@uni-konstanz.de}
\date{\today}

\begin{abstract}
 Based on the Chernoff approximation, we provide a general approximation result for convex 
 monotone semigroups which are continuous w.r.t. the mixed topology on suitable spaces of 
 continuous functions. Starting with a family $(I(t))_{t\geq 0}$ of operators, the semigroup 
 is constructed as the limit $S(t)f:=\lim_{n\to\infty}I(\nicefrac{t}{n})^n f$ and is uniquely 
 determined by the time derivative $I'(0)f$ for smooth functions. We identify explicit conditions 
 for the generating family $(I(t))_{t\geq 0}$ that are transferred to the semigroup $(S(t))_{t\geq 0}$
 and can easily be verified in applications. Furthermore, there is a structural link between 
 Chernoff type approximations for nonlinear semigroups and law of large numbers and central limit 
 theorem type results for convex expectations. The framework also includes large deviation results.

\end{abstract}

\maketitle

\section{Introduction}

In this article, we establish a general approximation result for convex monotone semigroups
which is based on the Chernoff approximation, see~\cite{BK,chernoff68, chernoff74}, that 
generalizes the Trotter--Kato product formula for linear semigroups, see~\cite{Kato,Trotter}.  
The idea is to start with a generating family $(I(t))_{t\geq 0}$ of operators 
$I(t)\colon\Cb\to\Cb$ which do not form a semigroup but have, for smooth functions, 
the derivative 
\[ I'(0)f:=\lim_{h\downarrow 0}\frac{I(h)f-f}{h}\in\Cb. \]
Here, the space $\Cb$ consists of all bounded continuous functions $f\colon\R^d\to\R$.
In order to obtain a corresponding semigroup $(S(t))_{t\geq 0}$ on $\Cb$ with prescribed 
generator $Af=I'(0)f$, we iterate the operators $(I(t))_{t\geq 0}$ over a sequence of
equidistant partitions with mesh size tending to zero and define the semigroup as the limit 
\begin{equation} \label{eq:intro}
 S(t)f:=\lim_{n\to\infty}I\big(\tfrac{t}{n}\big)^n f\in\Cb. 
\end{equation}
Recall that in the case of the Trotter--Kato formula the choice $I(t):=e^{tA_1}e^{tA_2}$ 
leads to a semigroup with generator $A=A_1+A_2$. There are two major questions
to address. 

First, we need a suitable topology for the convergence in equation~\eqref{eq:intro}. 
For Nisio semigroups, see~\cite{DKN20,NR21, Nisio}, it follows from
$I(s+t)f\leq I(s)I(t)f$ for all $s,t\geq 0$ and $f\in\Cb$ that the sequence 
$(I(2^{-n}t)^{2^n}f)_{n\in\N}$ is non-decreasing and the semigroup can be defined
pointwise as the monotone limit $S(t)f:=\sup_{n\in\N}I(2^{-n}t)^{2^n}f$. Typical 
examples have the form
\[ I(t)f:=\sup_{\lambda\in\Lambda}\big(S_\lambda(t)f-\phi(\lambda)t\big), \]
for a family $(S_\lambda)_{\lambda\in\Lambda}$ of linear semigroups $(S_\lambda(t))_{t\geq 0}$ 
and a function $\phi\colon\Lambda\to [0,\infty]$. It is also possible to define 
the semigroup as the limit of a decreasing sequence, see~\cite{BEK, FKN}, where the authors 
study semigroups corresponding to Markov processes with uncertain transition probabilities. 
While Chernoff's original work requires convergence w.r.t. a metric, convergence w.r.t. 
the supremum norm cannot be expected in the present setting. On the other hand,
it has been recently shown in~\cite{GNR} that the non-metrizable mixed topology 
between the supremum norm and the topology of uniform convergence of compacts
is suitable to study (linear) semigroups and characterize Markov processes via their 
infinitesimal generators. Note that a sequence $(f_n)_{n\in\N}$ converges in the mixed 
topology if and only if it is bounded w.r.t. the (weighted) supremum norm and converges 
uniformly on compacts. In particular, by Dini's theorem, the previously mentioned 
monotone convergence implies convergence in the mixed topology. Moreover, compactness 
in the mixed topology can be characterized by Arz\'ela-Ascoli's theorem, i.e., equicontinuity 
of the sequence $(I(\nicefrac{t}{n})^n f)_{n\in\N}$ rather than monotonicity. For the intended
application of the approximation result to limit theorems for convex expectations, it can not 
be guaranteed that the sequence $(I(\nicefrac{t}{n})^n f)_{n\in\N}$ is non-decreasing.

Second, since the convergence in equation~\eqref{eq:intro} is based on relative
compactness, we obtain a priori only convergence for a subsequence, i.e., 
\begin{equation} \label{eq:intro2}
 S(t)f:=\lim_{l\to\infty}I\big(\tfrac{t}{n_l}\big)^{n_l} f. 
\end{equation}
Following the arguments in~\cite{BDKN,BK}, one can only show that the convergent subsequence 
can be chosen independent of $f\in\Cb$ and $t\in\T$, where $\T\subset [0,\infty)$ is countable 
and dense. However, if the semigroup $(S(t))_{t\geq 0}$ was uniquely determined by its 
infinitesimal generator $Af=I'(0)f$, one could argue that the limit in equation~\eqref{eq:intro2} 
does not depend on the choice of the convergent subsequence and thus obtain the desired 
convergence in equation~\eqref{eq:intro}. In view of~\cite[Theorem~6.2]{GNR}, a possible
approach to show that the semigroup is uniquely determined by its generator could rely on
comparison principles for viscosity solutions of fully nonlinear equations, see, e.g.,~\cite{CIL, FS06}. 
In this article, we use instead the comparison principle from~\cite{BDKN} for strongly continuous  
convex monotone semigroups, where the generator is defined as a limit w.r.t. $\Gamma$-convergence
for functions in the upper Lipschitz set. It is important to note that the upper Lipschitz set is 
larger than the domain of the generator defined w.r.t. the mixed topology. In general, the latter 
is not invariant under the semigroup and the classical theory of m-accretive operators cannot be 
applied, see~\cite{BDKN} for a detailed discussion. 

By the comparison principle in~\cite{BDKN}, two strongly continuous convex monotone 
semigroups which are uniformly continuous w.r.t. the mixed topology coincide if they have 
the same upper Lipschitz set and their generators evaluated at smooth functions coincide. 
While the latter property can often be verified by straightforward computations, it remained 
an open problem how to show equality of the upper Lipschitz sets even in the case where 
the two semigroups are obtained by choosing different convergent subsequences in 
equation~\eqref{eq:intro2}. Hence, the first main contribution of this article is to provide 
explicit conditions under which two semigroups whose generators evaluated at smooth functions 
coincide also have the same upper Lipschitz set, see Theorem~\ref{thm:Lset}. By applying this 
result on semigroups which are obtained from choosing different convergent subsequences in
equation~\eqref{eq:intro2}, we obtain that the limit does not depend on the choice of the 
convergent subsequence. Hence, the desired convergence in equation~\eqref{eq:intro} follows and 
the semigroup $(S(t))_{t\geq 0}$ is uniquely determined by the derivative $I'(0)f$ for smooth 
functions $f$, see Theorem~\ref{thm:cher} and Corollary~\ref{cor:IJ}. The conditions required 
in Assumption~\ref{ass:cher} are explicit and can be verified for a variety of examples by 
straightforward computations. 
 
In Section~\ref{sec:LLN} and Section~\ref{sec:CLT}, we show that there is a structural link 
between Chernoff type approximations for convex monotone semigroups of the form~\eqref{eq:intro} 
and law of large numbers (LLN) and central limit theorem (CLT) type results for convex
expectations. The following diagram depicts this for the classical CLT: for every iid sequence 
$(\xi_n)_{n\in\N}$ with finite second moments, 
\begin{center}
 \begin{tikzpicture}
  \node (1) at (0,0) {$\e\left[f\left(\frac{1}{\sqrt{n}}\sum_{i=1}^n \xi_i\right)\right]$};
  \node[below=1cm of 1] (2) {$\int_{\R^d} f(y)\,\n(0,1)(\d y)$};
  \draw[->] (1) -- (2) node[midway,left,font=\small] {CLT};
  \node[right=0cm of 1] (3) {$\,=\; (I(\frac{1}{n})^n)(0)$};
  \node[right=0cm of 2] (4) {$=\; (S(1)f)(0)$,};
  \draw[->] (3) -- (4) node[midway,right,font=\small] {Chernoff};
 \end{tikzpicture}
\end{center}
where $(I(t)f)(x):=\e[f(x+\sqrt{t}\xi_1)]$ and $(S(t)f)(x):=\int_{\R^d}f(x+y)\,\n(0,t)(\d y)$ 
for all $t\geq 0$, $f\in\Cb$ and $x\in\R^d$. While in the linear case it is natural to 
use the convergence on the left-hand side to obtain the convergence on the right-hand side, 
in the nonlinear case we will use Theorem~\ref{thm:cher} to show the reverse implication
instead. To be precise: we will replace the linear expectation $\e[\,\cdot\,]$ by a convex 
expectation $\E[\,\cdot\,]$ and show that
\[ (S(1)f)(0)=\lim_{n\to\infty}\frac{1}{n}\E\left[nf\left(\frac{1}{\sqrt{n}}\sum_{i=1}^n \xi_i\right)\right]
	\quad\mbox{for all } f\in\Cb. \]
Furthermore, the generator of $(S(t))_{t\geq 0}$ given by 
\[ (Af)(x)=\E\left[\frac{1}{2}\xi_1^T D^2f(x)\xi_1\right] 
	\quad\mbox{for all } f\in\Cb^2 \mbox{ and } x\in\R^d \]
uniquely characterizes the semigroup, i.e., the limit is $G$-distributed with
\[ G\colon\S^d\to\R,\; a\mapsto\E\left[\frac{1}{2}\xi_1^T a\xi_1\right]. \]
This result is stated in Theorem~\ref{thm:CLT} and is illustrated by an
application to uncertain samples in Wasserstein spaces, see Theorem~\ref{thm:Wasser2}
and Theorem~\ref{thm:Wasser3}. For a brief overview on convex expectations,
we refer to Appendix~\ref{app:E}. To the best of our knowledge, this is also
the first result of this kind for convex rather than sublinear expectations.
In the sublinear case, Peng~\cite{Peng08} introduced the $G$-distribution by
\[ F_G\colon\Lipb\to\R,\; f\mapsto u^f(1,0), \]
where $u^f$ denotes the unique viscosity solution of the fully nonlinear PDE
\[ \begin{cases}
    \partial_t u(t,x)=G(D^2 u(x)), & (t,x)\in (0,\infty)\times\R^d, \\
    u(0,x)=f(x), & x\in\R^d
    \end{cases}\]
with $G(a):=\frac{1}{2}\sup_{\lambda\in\Lambda}\tr(\lambda\lambda^T a)$ for a bounded 
closed non-empty set $\Lambda\subset\R^{d\times d}$. The $G$-distribution is interpreted 
as a normal distribution with uncertain covariance and the corresponding CLT was proved
in~\cite{Peng08b,Peng10,Peng19}. We also want to mention extensions to L\'evy processes, 
see~\cite{BM, HJLP}, and results regarding convergence rates, see~\cite{HJL21, HL20, Krylov, Song}.
Note that the proofs given in~\cite{Peng08b,Peng19} rely heavily on a deep interior 
estimate of a fully nonlinear PDE while the more probabilistic approach in~\cite{Peng10} 
based on tightness and weak compactness requires an additional moment condition. 
Since~\cite[Theorem~6.2]{GNR} guarantees $F_G(f)=(S(1)f)(0)$ for all $f\in\Lipb$,
Theorem~\ref{thm:CLT} is consistent with previous results for the sublinear expectations.
Moreover, the proof of Theorem~\ref{thm:CLT} resembles the approach in~\cite{Peng10}, 
but covers the sublinear case under the more natural moment condition $\lim_{c\to\infty}\E[(|\xi_1|^2-c)^+]=0$ 
from~\cite{Peng19}.

Similarly, we can obtain LLN type results of the form
\[ (S(1)f)(0)=\lim_{n\to\infty}\frac{1}{n}\E\left[nf\left(\frac{1}{n}\sum_{i=1}^n \xi_i\right)\right]
	\quad\mbox{for all } f\in\Cb, \]
where the generator is given by $(Af)(x)=\E[\nabla f(x)\xi_1]$ for all $f\in\Cb^1$ and $x\in\R^d$.
This result is stated in Theorem~\ref{thm:LLN} and Theorem~\ref{thm:LLN2} shows that
the limit is maximally distributed, i.e., 
\[ (S(1)f)(0)=\sup_{x\in\R^d}\big(f(x)-\phi(y)\big) \quad\mbox{for all } f\in\Cb, \]
where $\phi(y):=\sup_{z\in\R^d}(yz-\E[z\xi])$ and $\xi:=\id_{\R^d}$.
In the sublinear case, this result goes back to Peng~\cite{Peng08b,Peng19} 
who introduced the maximal distribution by 
\[ F_\Lambda\colon\Lipb\to\R,\; f\mapsto\sup_{\lambda\in\Lambda}f(\lambda) \]
for a bounded closed non-empty set $\Lambda\subset\R^d$. There are two major
differences between the framework of this article and previous works.
First, we consider a sequence $(\psi_n)_{n\in\N}$ of recursively defined 
functions and are interested in the limit behaviour of 
\[ X_n:=\psi_n\big(\tfrac{1}{n}, 0, \xi_1,\ldots,\xi_n\big). \]
While the choice $\psi_n(\tfrac{1}{n}, 0, \xi_1,\ldots,\xi_n):=\frac{1}{n}\sum_{i=1}^n \xi_i$
is clearly admissible, the sample $\xi_{n+1}$ can also be randomly shifted by a
nonlinear function depending on the average of the previous samples $\xi_1,\ldots,\xi_n$,
see Corollary~\ref{cor:LLN}. Second, our results are not restricted to sublinear expectations 
and therefore create a uniform framework to cover both LLN type results for samples with 
uncertain distribution and large deviation results. For instance, Cram\'er's theorem 
can be seen as LLN for the entropic risk measure, see Corollary~\ref{cor:cramer}.
In view of this connection, the previously possibly unexpected scaling 
$\frac{1}{n}\E[nf(X_n)]$ instead of $\E[f(X_n)]$ suddenly appears to be very natural. 
We further extend Cram\'ers theorem to samples that are perturbed by a nonlinear 
function as in Corollary~\ref{cor:LLN}. In this case, the asymptotic convergence
rate is lowered according to the size of perturbation, see Theorem~\ref{thm:cramer}. 
Finally, we pick up the setting from Lacker~\cite{Lacker}, based on the weak convergence 
approach in~\cite{DE97}, who has previously worked in a similar setting with convex 
expectations and established a non-exponential extension of Sanov's theorem.
This leads to polynomial convergence rates that only require the existence of finite 
$p$-moments, see Theorem~\ref{thm:cramer2}. Large deviation theory based on suitable
classes of risk measures has recently also been explored by several other authors, 
see, e.g.,~\cite{BLT,Eckstein,FK,KZ}. The abstract results are again illustrated by 
uncertain samples in Wasserstein spaces, see Theorem~\ref{thm:Wasser}. Finally, 
we want to mention related LLN type results for capacities, see~\cite{MM}, 
and for coherent lower previsions, see~\cite{DM}.

The large deviation results obtained in Section~\ref{sec:LLN} show that the theory 
of nonlinear semigroups developed in Section~\ref{sec:SG} and the previous works~\cite{BDKN,BK}
can be used to obtain asymptotic convergence rates. Furthermore, it is even possible 
to derive non-asymptotic rates for the limit in equation~\eqref{eq:intro} which 
yields error bounds for the LLN and CLT type results in Section~\ref{sec:LLN}
and Section~\ref{sec:CLT}. These results require additional moment conditions 
and are consistent with the ones in~\cite{HJL21, HL20, Krylov, Song} for 
sublinear expectations. For details, we refer to~\cite{BJKL}.

\section{Convex monotone semigroups}
\label{sec:SG}

The results in this section rely strongly on the previous works~\cite{BDKN,BK} which have,
despite their generality, two major drawbacks. First, we only obtain a convergent subsequence
that exists due to a relative compactness argument. Furthermore, the comparison principle
in~\cite{BDKN} uniquely characterizes semigroups by their upper $\Gamma$-generator
defined on their upper Lipschitz set. This theoretical result is an analogue to the classical 
statement that strongly continuous linear semigroups are uniquely determined by their 
generators defined on their maximal domain but the upper $\Gamma$-generator and the upper 
Lipschitz can typically not be computed explicitly. Hence, we are looking for sufficient 
conditions under which nonlinear semigroups are uniquely determined by their generators evaluated 
at smooth functions. This resembles the concept of a core in the theory of linear semigroups. 
While~\cite[Section~3]{BDKN} provides results that allow to approximate the $\Gamma$-generator 
with smooth functions, the question whether two semigroups have the same upper Lipschitz set 
remained open. Theorem~\ref{thm:Lset} now shows that, under conditions which are consistent 
with the framework in~\cite{BDKN}, two semigroups whose generators evaluated at smooth 
functions coincide have the same upper Lipschitz set. In view of the results in~\cite{BDKN}
these semigroups are then uniquely determined by their generators evaluated at smooth
functions, see Theorem~\ref{thm:comp}. Moreover, we provide sufficient conditions on the family
$(I(t))_{t\geq 0}$ that transfer to the corresponding semigroup given by equation~\eqref{eq:intro2} 
so that Theorem~\ref{thm:comp} applies and equation~\eqref{eq:intro} is valid.
Throughout this article, we consider semigroups that are defined 
on spaces of continuous functions which do not exceed a certain growth rate at infinity. 
For that purpose, let $\kappa\colon\R^d\to (0,\infty)$ be a bounded continuous function with
\begin{equation} \label{eq:kappa}
 c_\kappa:=\sup_{x\in\R^d}\sup_{|y|\leq 1}\frac{\kappa(x)}{\kappa(x-y)}<\infty
\end{equation} 
and denote by $\Ck$ the space of all continuous functions $f\colon\R^d\to\R$ with
\[ \|f\|_\kappa:=\sup_{x\in\R^d}|f(x)|\kappa(x)<\infty. \]
In Section~\ref{sec:LLN} and Section~\ref{sec:CLT}, we will choose $\kappa(x):=(1+|x|^p)^{-1}$ 
for $p=1$ and $p=2$, respectively. Moreover, we endow $\Ck$ with the mixed topology 
between $\|\cdot\|_\kappa$ and the topology of uniform convergence on compacts sets. 
It is well-known, see~\cite[Proposition~A.4]{GNR}, that a sequence $(f_n)_{n\in\N}\subset\Ck$ 
converges to $f\in\Ck$ w.r.t. the mixed topology if and only if 
\[ \sup_{n\in\N}\|f_n\|_\kappa<\infty \quad\mbox{and}\quad
	\lim_{n\to\infty}\|f-f_n\|_{\infty,K}=0 \]
for all compact subsets $K\subset\R^d$, where $\|f\|_{\infty,K}:=\sup_{x\in K}|f(x)|$. 
Similarly, for a family $(f_h)_{h>0}\subset\Ck$ and $f\in\Ck$, it holds that $f=\lim_{h\downarrow 0}f_h$
w.r.t. the mixed topology if and only if $\sup_{h\in (0,h_0]}\|f_h\|_\infty<\infty$ for some 
$h_0>0$ and, for every $\epsilon>0$ and compact $K\subset\R^d$, there exists $h_\epsilon>0$ with 
$\|f-f_h\|_{\infty,K}<\epsilon$ for all $h\in (0,h_\epsilon]$. 
Subsequently, if not stated otherwise, all limits in $\Ck$ are understood w.r.t. the mixed topology 
and compact subsets are denoted by $K\Subset\R^d$. For more details, we refer to~\cite{GNR} 
and the references therein. Moreover, the set of all real-valued functions is endowed with the 
pointwise order, i.e., for functions $f,g\colon\R^d\to\R$ we write $f\leq g$ if and only if 
$f(x)\leq g(x)$ for all $x\in\R^d$. All order-related notions (sup, inf, max, min, $\limsup$, etc.) 
for such functions are understood with respect to this order and the positive part of a function 
$f\colon\R^d\to\R$ is denoted by $f^+:=\max\{f,0\}$. Let
\[ B_{\Ck}(r):=\{f\in\Ck\colon\|f\|_\infty\leq r\} \quad\mbox{and}\quad
	B_{\R^d}(r):=\{x\in\R^d\colon |x|\leq r\} \]
be the closed balls with radius $r\geq 0$ around zero, where $|\,\cdot\,|$ denotes the 
Euclidean distance. The space $\Cb$ consists of all 
bounded continuous functions $f\colon\R^d\to\R$ and $\Cbi$ is the space of 
all infinitely differentiable functions $f\in\Cb$ such that all partial derivatives are in $\Cb$. 
Let $\Lipb$ be the space of all bounded Lipschitz continuous functions 
$f\colon\R^d\to\R$ and, for every $r\geq 0$, the set $\Lipb(r)$ consists of all bounded
$r$-Lipschitz functions $f\colon\R^d\to\R$ with $\|f\|_\infty\leq r$. Here, 
$\|f\|_\infty:=\sup_{x\in\R^d}|f(x)|$ denotes the usual supremum norm. Let
\[ (\tau_x f)(y):=f(x+y) \quad\mbox{for all } x,y\in\R^d \mbox{ and } f\colon\R^d\to\R \]
be the shift operators and $\R_+:=\{x\in\R\colon x\geq 0\}$ the positive real numbers 
including zero.

\subsection{The upper Lipschitz set}
\label{sec:Lset}

We start with a formal definition of the (upper) Lipschitz set and some basic terminology
concerning nonlinear semigroups. Then, after an auxiliary lemma, we give explicit conditions 
which ensure that two semigroups whose generators coincide for smooth functions also 
have the same upper Lipschitz set. Recall that limits (and thus continuity) in $\Ck$ are 
understood w.r.t. the mixed topology rather than the norm $\|\cdot\|_\kappa$.

\begin{definition}	
 Let $(I(t))_{t\geq 0}$ be a family of operators $I(t)\colon\Ck\to\Ck$ with $I(0)f=f$ for all 
 $f\in\Ck$. The Lipschitz set $\L^I$ consists of all $f\in\Ck$ such that there exist $c\geq 0$ 
 and $t_0>0$ with
 \[ \|I(t)f-f\|_\kappa\leq ct \quad\mbox{for all } t\in [0,t_0]. \]
 The upper Lipschitz set $\L^I_+$ contains all $f\in\Ck$ such that there exist $c\geq 0$ and 
 $t_0>0$ with
 \[ \|(I(t)f-f)^+\|_\kappa\leq ct \quad\mbox{for all } t\in [0,t_0]. \]
 For $f\in\Ck$ such that the following limit exists and lies in $\Ck$, 
 we define the derivative of the mapping $t\mapsto I(t)f$ at zero by
 \[ I'(0)f:=\lim_{h\downarrow 0}\frac{I(h)f-f}{h}. \]
\end{definition}

While the constant $c\geq 0$ in the previous definition typically depends on $f$,
the parameter $t_0>0$ can often be chosen arbitrarily. Furthermore, since convergence w.r.t. 
the mixed topology incorporates norm boundedness, it holds that $f\in\L^I$ for all $f\in\Ck$ 
such that the limit $I'(0)f$ exists.

\begin{definition}
 A family $(S(t))_{t\geq 0}$ of operators $S(t)\colon\Ck\to\Ck$ is called a semigroup 
 if $S(0)=\id_{\Ck}$ and $S(s+t)f=S(s)S(t)f$ for all $s,t\geq 0$ and $f\in\Ck$.
 The semigroup is called convex (monotone) if the mapping
 $\Ck\to\R,\; f\mapsto (S(t)f)(x)$ is convex (monotone) for all $t\geq 0$ and $x\in\R^d$. 
 Moreover, the semigroup is called strongly continuous if the mapping 
 $\R_+\to\Ck,\; t\mapsto S(t)f$ is continuous for all $f\in\Ck$. 
 The generator is defined by 
 \[ A\colon D(A)\to\Ck,\; f\mapsto\lim_{h\downarrow 0}\frac{S(h)f-f}{h}, \] 
 where the domain $D(A)$ consists of all $f\in\Ck$ such that the previous limit exists.
\end{definition}

Since convergence w.r.t. the mixed topology incorporates norm boundedness, we obtain 
$D(A)\subset\L^S$. For every $t\geq 0$, $f\in\Ck$ and $x\in\R^d$ such that the mapping 
$s\mapsto (S(s)f)(x)$ is measurable and bounded from above, we define the pointwise integral
\[ \left(\int_0^t S(s)f\,\d s\right)(x):=\int_0^t (S(s)f)(x)\,\d s. \]

\begin{lemma} \label{lem:Lset}
 Let $(S(t))_{t\geq 0}$ be a convex semigroup on $\Ck$ such that $S(t)\colon\Ck\to\Ck$ 
 is continuous for all $t\geq 0$. For every $r,t\geq 0$, we assume that there exists $c\geq 0$ with 
 \begin{equation} \label{eq:Lset}
  \|S(s)f-S(s)g\|_\kappa\leq c\|f-g\|_\kappa
 \end{equation}
 for all $s\in [0,t]$ and $f,g\in B_{\Ck}(r)$. Then, for every $t\geq 0$ and $f\in D(A)$, 
 \[ S(t)f-f\leq\int_0^t S(s)(f+Af)-S(s)f\,\d s. \]
\end{lemma}
\begin{proof}
 For every $x\in\R^d$, it follows from $f\in D(A)\subset\L^S$ and inequality~\eqref{eq:Lset}
 that the mapping $\R_+\to\R,\; t\mapsto (S(t)f)(x)$ is locally Lipschitz continuous and 
 thus differentiable almost everywhere. Hence, by Rademacher's theorem, 
 \[ (S(t)f-f)(x)=\int_0^t \frac{\d}{\d s}(S(s)f)(x)\,\d s \quad\mbox{for all } t\geq 0. \]
 For every $t\geq 0$ and $h\in (0,1]$, Lemma~\ref{lem:lambda} implies
 \begin{align*}
  &\frac{S(t)S(h)f-S(t)f}{h}-S(t)(f+Af)+S(t)f \\
  &\leq S(t)\left(\frac{S(h)f-f}{h}+f\right)-S(t)(f+Af) \\
  &\leq\frac{1}{2}S(t)\left(2\left(\frac{S(h)f-f}{h}-Af\right)+f+Af\right)
  -\frac{1}{2}S(t)(f+Af).
 \end{align*}
 Since $S(t)$ is continuous, the right-hand side converges to zero as $h\downarrow 0$ 
 which yields the claim. 
\end{proof}

\begin{assumption} \label{ass:Lset}
 Let $(S(t))_{t\geq 0}$ be a convex monotone semigroup on $\Ck$ which satisfies 
 the following conditions:
 \begin{itemize}
  \item[(i)] For every $r,t\geq 0$, there exists $c\geq 0$ with 
   \[ \|S(s)f-S(s)g\|_\kappa\leq c\|f-g\|_\kappa \]
   for all $s\in [0,t]$ and $f,g\in B_{\Ck}(r)$. Moreover, it holds that $S(t)0=0$ for all $t\geq 0$. 
  \item[(ii)] The operator $S(t)\colon\Ck\to\Ck$ is continuous for all $t\geq 0$.
  \item[(iii)] For every $f\in\L^S_+\cap\Lipb$, there exist $L\geq 0$, $t_0>0$ and $\delta>0$ 
   with 
   \[ \|S(t)(\tau_x f)-\tau_x S(t)f\|_\kappa\leq Lt \] 
   for all $t\in [0,t_0]$ and $x\in B_{\R^d}(\delta)$. 
 \end{itemize}
\end{assumption}

Condition~(i) means that the semigroup is locally uniformly continuous w.r.t. the 
weighted supremum norm while condition~(ii) states that the semigroup is continuous 
at a fixed time w.r.t. the mixed topology. Condition~(iii) is, in particular, satisfied if the
semigroup is translation-invariant, i.e., $S(t)(\tau_x f)=\tau_x S(t)f$ for all $t\geq 0$, 
$f\in\Ck$ and $x\in\R^d$.

\begin{theorem} \label{thm:Lset}
 Let $(S(t))_{t\geq 0}$ and $(T(t))_{t\geq 0}$ be two semigroups satisfying
 Assumption~\ref{ass:Lset}. Moreover, we assume that $\Cb^\infty\subset D(A)\cap D(B)$
 and $(Af)^+\leq (Bf)^+$ for all $f\in\Cb^\infty$, where $A$ and $B$ denote the generators of
 $(S(t))_{t\geq 0}$ and $(T(t))_{t\geq 0}$, respectively. Then, 
 \[ \L^T_+\cap\Lipb\subset\L^S_+\cap\Lipb. \]
\end{theorem}
\begin{proof}
 Fix $f\in\L^T_+\cap\Lipb$ and choose $L\geq 0$, $t_0>0$ and $\delta\in (0,1]$ satisfying
 Assumption~\ref{ass:Lset}(iii). Let $\eta\colon\R^d\to\R_+$ be an infinitely 
 differentiable function with $\supp(\eta)\subset B_{\R^d}(\delta)$ and $\int_{\R^d}\eta(x)\,\d x=1$. 
 For every $n\in\N$ and $x\in\R^d$, we define $\eta_n(x):=n^d\eta(nx)$ and 
 \[ f_n(x):=(f*\eta_n)(x)=\int_{\R^d} f(x-y)\eta_n(y)\,\d y. \]
 For every $t\in [0,t_0]$, $n\in\N$ and $x\in\R^d$, we use 
 Jensen's inequality, Assumption~\ref{ass:Lset}(iii) and inequality~\eqref{eq:kappa} to 
 estimate
 \begin{align*}
  (T(t)f_n-f_n)(x)\kappa(x)
  &\leq\int_{B(\delta)}\big(T(t)(\tau_{-y}f)-\tau_{-y}f\big)(x)\kappa(x)\eta_n(y)\,\d y \\
  &\leq\int_{B(\delta)}(T(t)f-f)(x-y)\kappa(x)\eta_n(y)\,\d y+Lt \\ 
  &\leq c_\kappa\int_{B(\delta)}(T(t)f-f)(x-y)\kappa(x-y)\eta_n(y)\,\d y+Lt \\
  &\leq c_\kappa\big((T(t)f-f)^+\kappa\big)*\eta_n+Lt, 
 \end{align*}
 where $B(\delta):=B_{\R^d}(\delta)$. Taking the supremum over $x\in\R^d$ yields
 \[ \|(T(t)f_n-f_n)^+\|_\kappa\leq c_\kappa\|(T(t)f-f)^+\|_\kappa+Lt. \]
 Hence, it follows from $f\in\L^T_+$ that there exist $c\geq 0$ and $t_1\in (0,t_0]$ with
 \[ \|(T(t)f_n-f_n)^+\|_\kappa\leq (cc_\kappa+L)t
 	\quad\mbox{for all } n\in\N \mbox{ and } t\in [0,t_1]. \]
 Since $f_n\in\Cbi$, we obtain $\|(Af_n)^+\|_\kappa\leq\|(Bf_n)^+\|_\kappa\leq cc_\kappa+L$ 
 for all $n\in\N$. Moreover, inequality~\eqref{eq:kappa} guarantees that 
 $\|f_n\|_\kappa\leq c_\kappa\|f\|_\kappa$ for all $n\in\N$. Lemma~\ref{lem:Lset}, 
 the monotonicity of $S(s)$, Assumption~\ref{ass:Lset}(i) and inequality~\eqref{eq:kappa} 
 imply that there exists $c'\geq 0$ with
 \begin{align*}
  (S(t)f_n-f_n)\kappa 
  &\leq\int_0^t \big(S(s)(f_n+(Af_n)^+)-S(s)f_n\big)\kappa\,\d s \\
  &\leq\int_0^t c'\|(Af_n)^+\|_\kappa\,\d s \leq c'\big(cc_\kappa+L\big)t
 \end{align*}
 for all $n\in\N$ and $t\in [0,t_1]$. Taking the limit $n\to\infty$ yields
 \[ \|(S(t)f-f)^+\|_\kappa\leq c'\big(cc_\kappa+L\big)t \quad\mbox{for all } t\in [0,t_1] \]
 which shows that $\L^T_+\cap\Lipb\subset\L^S_+\cap\Lipb$. 
\end{proof}

While the following comparison principle does not have the same generality 
as~\cite[Theorem~2.10]{BDKN}, it has the advantage that the conditions~(i)--(v)
can easily be verified in many applications and that the semigroups are uniquely
determined by their generators evaluated at smooth functions. Moreover, 
Assumption~\ref{ass:cher} provides sufficient conditions on the generating family 
$(I(t))_{t\geq 0}$ such that Theorem~\ref{thm:comp} applies to the corresponding 
semigroup $(S(t))_{t\geq 0}$. In Subsection~\ref{sec:cher}, we also discuss how 
the conditions~(i)--(v) can be verified.

\begin{theorem} \label{thm:comp}
 Let $(S_1(t))_{t\geq 0}$ and $(S_2(t))_{t\geq 0}$ be two strongly continuous convex 
 monotone semigroups on $\Ck$ with $S_i(t)0=0$ and generators $A_1$ and $A_2$, 
 respectively, such that the following conditions are satisfied for $i=1,2$:
 \begin{itemize}
  \item[(i)] It holds $\Cbi\subset D(A_i)$ and $A_1 f=A_2 f$ for all $f\in\Cbi$. 
  \item[(ii)] For every $r,T\geq 0$, there exists $c\geq 0$ with
   \[ \|S_i(t)f-S_i(t)g\|_\kappa\leq c\|f-g\|_\kappa 
   	\quad\mbox{for all } t\in [0,T] \mbox{ and } f,g\in B_{\Ck}(r). \]
  \item[(iii)] For every $\epsilon>0$, $r,T\geq 0$ and $K\Subset\R^d$, there exist $c\geq 0$ 
   and $K'\Subset\R^d$ with
   \[ \|S_i(t)f-S_i(t)g\|_{\infty,K}\leq c\|f-g\|_{\infty,K'}+\epsilon \]
   for all $t\in [0,T]$ and $f,g\in B_{\Ck}(r)$. 
  \item[(iv)] For every $f\in\Lipb$ and $\epsilon>0$, there exist $\delta, t_0>0$ with
   \[ \|S_i(t)(\tau_x f)-\tau_x S_i(t)f\|_\kappa\leq\epsilon t \]
   for all $t\in [0,t_0]$ and $x\in B_{\R^d}(\delta)$.
  \item[(v)] It holds $S_i(t)\colon\Lipb\to\Lipb$ for all $t\geq 0$. 
 \end{itemize}
 Then, it holds $S_1(t)f=S_2(t)f$ for all $t\geq 0$ and $f\in\Ck$. 
\end{theorem}
\begin{proof}
 It follows from Theorem~\ref{thm:Lset} that
 \begin{equation} \label{eq:LS+}
  \L^{S_1}_+\cap\Lipb=\L^{S_2}_+\cap\Lipb
 \end{equation}
 and therefore~\cite[Theorem~2.10]{BDKN} implies
 \begin{equation} \label{eq:Sn2}
  S_1(t)f=S_2(t)f \quad\mbox{for all } (f,t)\in\Cbi\times\R_+.
 \end{equation}
 Indeed, the semigroups $(S_1(t))_{t\geq 0}$ and $(S_2(t))_{t\geq 0}$ clearly
 satisfy~\cite[Assumption~2.4]{BDKN} while equation~\eqref{eq:LS+}, condition~(ii) 
 and condition~(v) imply $\L^{S_1}\cap\Lipb\subset\L^{S_2}_+$
 and 
 \[ S_1(t)\colon\L^{S_1}\cap\Lipb\to\L^{S_1}\cap\Lipb
 	\quad\mbox{for all } t\geq 0. \]
 Moreover, due to condition~(i) and~(iv), we can apply~\cite[Lemma~3.6]{BDKN} 
 to conclude that $(S_1(t))_{t\geq 0}$ and $(S_2(t))_{t\geq 0}$ have the same upper 
 $\Gamma$-generator on $\L^{S_1}\cap\Lipb$. Hence, 
 \[ S_1(t)f\leq S_2(t)f \quad\mbox{for all } t\geq 0 \mbox{ and } f\in\Cbi \]
 and reversing the roles of $(S_1(t))_{t\geq 0}$ and $(S_2(t))_{t\geq 0}$ yields that 
 equation~\eqref{eq:Sn2} is valid. Since $\Cbi\subset\Ck$ is dense, it 
 follows from condition~(iii) that 
 \[ S_1(t)f=S_2(t)f \quad\mbox{for all } (f,t)\in\Ck\times\R_+. \qedhere \]
\end{proof}

\subsection{The Chernoff approximation}
\label{sec:cher}

Due to Theorem~\ref{thm:Lset} and the resulting Theorem~\ref{thm:comp}, we are now able to
improve the approximation result for convex monotone semigroups from~\cite{BDKN}.
Let $(I(t))_{t\geq 0}$ be a family of operators $I(t)\colon\Ck\to\Ck$ and $(h_n)_{n\in\N}\subset (0,1]$ 
be a sequence with $h_n\to 0$. For every $t\geq 0$, $n\in\N$ and $f\in\Ck$, we define 
\[ I(\pi^t_n)f:=I(h_n)^{k_n^t}f=\underbrace{\big(I(h_n)\circ\ldots\circ I(h_n)\big)}_{k_n^t \text{ times}}f, \]
where $k_n^t:=\max\{k\in\N_0\colon kh_n\leq t\}$ and $\pi_n^t:=\{kh_n\wedge t\colon k\in\N_0\}$
denotes the equidistant partition of $[0,t]$ with mesh size $h_n$.

\begin{assumption} \label{ass:cher}
 Let $(I(t))_{t\geq 0}$ be a family of operators $I(t)\colon\Ck\to\Ck$ which satisfy 
 the following conditions:
 \begin{itemize}
  \item[(i)] $I(0)=\id_{\Ck}$.
  \item[(ii)] $I(t)$ is convex and monotone with $I(t)0=0$ for all $t\geq 0$.
  \item[(iii)] There exists $\omega\geq 0$ with
   \[ \|I(t)f-I(t)g\|_\kappa\leq e^{\omega t}\|f-g\|_\kappa
   	\quad\mbox{for all } t\in [0,1] \mbox{ and } f,g\in\Ck. \]
  \item[(iv)] There exist $t_0>0$, $\delta\in (0,1]$ and $L\geq 0$ with
   \[ \|I(t)(\tau_x f)-\tau_x I(t)f\|_\kappa\leq Lrt|x| \]
   for all $t\in [0,t_0]$, $x\in B_{\R^d}(\delta)$, $r\geq 0$ and $f\in\Lipb(r)$. 
  \item[(v)] It holds that $\Cbi\subset\L^I$ and $I'(0)f\in\Ck$ exists for all $f\in\Cbi$. 
  \item[(vi)] For every $\epsilon>0$, $r,T\geq 0$ and $K\Subset\R^d$,
   there exist $c\geq 0$ and $K'\Subset\R^d$ with
   \[ \|I(\pi_n^t)f-I(\pi_n^t)g\|_{\infty,K}\leq\|f-g\|_{\infty,K'}+\epsilon \]
   for all $t\in [0,T]$, $f,g\in B_{\Ck}(r)$ and $n\in\N$.
  \item[(vii)] It holds that $I(t)\colon\Lipb(r)\to\Lipb(e^{\omega t}r)$ for all $r,t\geq 0$. 
 \end{itemize}
\end{assumption}

The key idea here is to identify conditions on the one-step operators $I(t)$ that are 
preserved during the iteration and thus transfer to the associated semigroup $(S(t))_{t\geq 0}$.
Although condition~(vi) is stated for the iterated operators $I(\pi_n^t)$,
there are several sufficient conditions for the one-step operators $I(t)$ 
that can be verified in applications. One of them is given in Remark~\ref{rem:cher}(b)
and such conditions are more systematically studied in~\cite[Subsection~2.5]{BKN}.
In many applications such as the ones presented in Section~\ref{sec:LLN} and
Section~\ref{sec:CLT}, most of the previous conditions can easily be verified
while the computation of the derivative $I'(0)f$ usually takes some computational
effort. However, for smooth functions this can often be achieved by elementary arguments
such as Taylor's formula with a suitable remainder estimate. In Theorem~\ref{thm:LLN}
and Theorem~\ref{thm:CLT}, we apply Theorem~\ref{thm:cher} with 
$(I(t)f)(x):=\frac{1}{t}\E[tf(x+t\xi_1)]$ and $(I(t)f)(x):=\frac{1}{t}\E[tf(x+\sqrt{t}\xi_1)]$,
respectively. The corresponding proofs consist in verifying Assumption~\ref{ass:cher}, 
where the most of the conditions are immediately satisfied and the computation of the 
derivative $I'(0)f$ is the main part of the proof. Furthermore, the need to satisfy 
Assumption~\ref{ass:cher} requires the random variables $(\xi_n)_{n\in\N}$ to satisfy 
certain natural moment conditions that also become apparent in the proofs. 
Further examples that involve similar computations can be found in~\cite[Section~5]{BDKN} 
and~\cite[Section~6]{BK}. Since monotone convergence implies convergence w.r.t. the
mixed topology, the Nisio semigroups studied in~\cite{DKN20,NR21} are also covered by 
the results of this section. Finally, we want to emphasize that, in contrast to the 
previous applications presented in~\cite[Section~5]{BDKN} and~\cite[Section~6]{BK}, 
the uniqueness of the semigroup $(S(t))_{t\geq 0}$ follows immediately now. 
In Theorem~\ref{thm:LLN2}, this allows us to explicitly represent the semigroup by 
the Hopf--Lax formula.

\begin{remark} \label{rem:cher}
 We want to discuss some of the previous conditions in more detail.
 \begin{itemize}
  \item[(a)] Condition~(iii) guarantees that the iterated operators $I(\pi_n^t)$ are uniformly 
   Lipschitz continuous. In many examples, one can choose $\omega:=0$.
  \item[(b)] Condition~(vi) can be verified as follows: suppose that there exists another 
   bounded continuous function $\tilde{\kappa}\colon\R^d\to (0,\infty)$ such that, 
   for every $\epsilon>0$, there exists $K\Subset\R^d$ with 
   $\sup_{x\in K^c}\frac{\tilde{\kappa}(x)}{\kappa(x)}\leq\epsilon$. 
   Moreover, we assume that there exists $c\geq 0$ with 
   \begin{equation} \label{eq:tilde}
    \|I(t)f\|_{\tilde{\kappa}}\leq e^{ct}\|f\|_{\tilde{\kappa}}
   \end{equation}
   for all $t\in [0,1]$ and $f\in\Ck$ with $\|f\|_{\tilde{\kappa}}\leq 1$. By induction, 
   one can show that 
   \[ \|I(\pi_n^t)f\|_{\tilde{\kappa}}\leq e^{ct}\|f\|_{\tilde{\kappa}} \]
   for all $t\geq 0$, $n\in\N$ and $f\in\Ck$ with $\|f\|_{\tilde{\kappa}}\leq e^{-ct}$. 
   Let $(f_n)_{n\in\N}\subset\Ck$ be a sequence with $f_n\downarrow 0$. 
   For every $\epsilon>0$, there exists $K\Subset\R^d$ with
   \[ |f_n(x)|\tilde{\kappa}(x)
   	=|f_n(x)|\kappa(x)\frac{\tilde{\kappa}(x)}{\kappa(x)}
   	\leq\|f_1\|_\kappa\frac{\tilde{\kappa}(x)}{\kappa(x)}<\epsilon \]
   for all $n\in\N$ and $x\in K^c$. Moreover, by Dini's theorem, there exists 
   $n_0\in\N$ with 
   \[ |f_n(x)|\tilde{\kappa}(x)\leq\|f_n\|_{\infty,K}\|\tilde{\kappa}\|_\infty<\epsilon \]
   for all $n\geq n_0$ and $x\in K$. This implies
   \[ \sup_{s\in [0,t]}\sup_{k\in\N}\|I(\pi_k^s)f_n\|_{\tilde{\kappa}}
   	\leq e^{ct}\|f_n\|_{\tilde{\kappa}}\to 0 \quad\mbox{as } n\to\infty. \]
   We obtain $\sup_{(t,x)\in [0,T]\times K}\sup_{k\in\N}\big(I(\pi_k^t)f_n\big)(x)\downarrow 0$
   as $n\to\infty$ for all $T\geq 0$ and $K\Subset\R^d$ and~\cite[Lemma~C.2]{BDKN} yields that 
   condition~(vi) is satisfied. In the subsequent sections, we choose $\tilde{\kappa}(x):=(1+|x|^p)^{-1}$ 
   for some $p\geq 1$ and define $I(t)$ via a nonlinear expectation. In this case, 
   inequality~\eqref{eq:tilde} is a moment condition. 
  \item[(c)] If $\kappa\equiv 1$, the conditions~(iii) and~(iv) imply 
   \[ I(t)\colon\Lipb(r)\to\Lipb(e^{(\omega+L)t}r) \quad\mbox{for all } r,t\geq 0. \]
 \end{itemize}
\end{remark}

\begin{theorem} \label{thm:cher}
 Let $(I(t))_{t\geq 0}$ be a family of operators satisfying Assumption~\ref{ass:cher}.
 Then, there exists a strongly continuous convex monotone semigroup $(S(t))_{t\geq 0}$ 
 on $\Ck$ with
 \begin{equation} \label{eq:cher}
  S(t)f=\lim_{n\to\infty}I(\pi_n^t)f \quad\mbox{for all } t\geq 0 \mbox{ and } f\in\Ck. 
 \end{equation}
 Furthermore, the semigroup has the following properties:
 \begin{itemize}
  \item[(i)] It holds that $f\in D(A)$ and $Af=I'(0)f$ for all $f\in\Ck$ such that $I'(0)f\in\Ck$ exists. 
   In particular, this is valid for all $f\in\Cbi$.
  \item[(ii)] It holds that $\|S(t)f-S(t)g\|_\kappa\leq e^{\omega t}\|f-g\|_\kappa$ for all $t\geq 0$ 
   and $f,g\in\Ck$.
  \item[(iii)] For every $\epsilon>0$, $r,T\geq 0$ and $K\Subset\R^d$, there exist 
   $K'\Subset\R^d$ and $c\geq 0$ with
   \[ \|S(t)f-S(t)g\|_{\infty,K}\leq c\|f-g\|_{\infty,K'}+\epsilon \]
   for all $t\in [0,T]$ and $f,g\in B_{\Ck}(r)$.  
  \item[(iv)] It holds that $\L^I\subset\L^S$ and $\L^I_+\subset\L^S_+$. Moreover, for every $t\geq 0$, 
   \[ S(t)\colon\L^S\to\L^S \quad\mbox{and}\quad S(t)\colon\L^S_+\to\L^S_+. \] 
  \item[(v)] For every $r, t\geq 0$, $f\in\Lipb(r)$ and $x\in B_{\R^d}(\delta)$,
   \[ \|S(t)(\tau_x f)-\tau_x S(t)f\|_\kappa\leq Lrte^{2\omega t}|x|. \]
   Furthermore, it holds that $S(t)\colon\Lipb(r)\to\Lipb(e^{\omega t}r)$ for all $r,t\geq 0$. 
 \end{itemize}
\end{theorem}
\begin{proof}
 First, we verify the assumptions of~\cite[Section~4]{BDKN}. Assumption~\ref{ass:cher}
 guarantees that the conditions~(i)-(iv) and~(vi) of~\cite[Assumption~4.1]{BDKN} 
 and~\cite[Assumption~4.4]{BDKN} are satisfied. Moreover, by induction, we obtain
 \begin{equation} \label{eq:Lipb}
  I(\pi_n^t)\colon\Lipb(r)\to\Lipb(e^{\omega t}r)
  \quad\mbox{for all } r,t\geq 0 \mbox{ and } n\in\N. 
 \end{equation}
 Since $\Cbi\subset\L^I\subset\Cb$ is separable and dense, this shows 
 that~\cite[Assumption~4.1(v)]{BDKN} is satisfied. Now, let $\T\subset\R_+$ be a 
 countable dense set including zero. By~\cite[Theorem~4.3]{BDKN}, there exist strongly
 continuous convex monotone semigroup $(S(t))_{t\geq 0}$ on $\Ck$ and a subsequence 
 $(n_l)_{l\in\N}\subset\N$ with
 \begin{equation} \label{eq:cher2}
  S(t)f=\lim_{l\to\infty}I(\pi_{n_l}^t)f \quad\mbox{for all } (f,t)\in\Ck\times\T. 
 \end{equation}
 Moreover, the semigroup has the properties~(i)-(iii) and it holds that $\L^I\subset\L^S$
 while property~(v) follows from~\cite[Theorem~4.5]{BDKN}. For the invariance of the 
 (upper) Lipschitz set, we refer to~\cite[Lemma~2.8]{BDKN}. In order to show
 $\L^I_+\subset\L^S_+$, let $f\in\L^I_+$ and choose $c\geq 0$ and $t_0>0$ with
 \[ \|(I(t)f-f)^+\|_\kappa\leq ct \quad\mbox{for all } t\in [0,t_0]. \]
 By induction, we show that, for every $k\in\N$ and $n\in\N$ with $h_n\leq t_0$, 
 \begin{equation} \label{eq:Lset+}
  \|(I(h_n)^k f-f)^+\|_\kappa\leq ckh_n e^{\omega kh_n}. 
 \end{equation}
 For $k=1$, the previous estimate holds by assumption. Suppose that inequality~\eqref{eq:Lset+}
 holds for some fixed $k\in\N$. Lemma~\ref{lem:lambda} and the monotonicity of 
 $I(h_n)^k$ yield
 \begin{align*}
  I(h_n)^{k+1}f-f &=I(h_n)^kI(h_n)f-I(h_n)^k f+I(h_n)^k f-f \\
  &\leq h_n\left(I(h_n)^k\left(\frac{(I(h_n)f-f)^+}{h_n}+f\right)-I(h_n)f\right)
  +I(h_n)^k f-f.
 \end{align*}
 Hence, it follows from Assumption~\ref{ass:cher}(iii),~\cite[Lemma~2.7]{BK} and 
 inequality~\eqref{eq:Lset+} that
 \begin{align*}
  \|(I(h_n)^{k+1}f-f)^+\|_\kappa
  &\leq e^{\omega kh_n}\|(I(h_n)f-f)^+\|_\kappa+\|(I(h_n)^k f-f)^+\|_\kappa \\
  &\leq ch_n e^{\omega kh_n}+ckh_n e^{\omega kh_n}
  \leq c(k+1)h_n e^{\omega(k+1)h_n}.
 \end{align*}
 Due to Equation~\eqref{eq:cher2} and the strong continuity, inequality~\eqref{eq:Lset+} 
 transfers to $S(t)$, i.e., 
 \[ \|(S(t)f-f)^+\|_\kappa\leq cte^{\omega t} \quad\mbox{for all } t\geq 0. \]
 
 Second, we verify equation~\eqref{eq:cher} by showing that the limit in
 equation~\eqref{eq:cher2} does not depend on the choice of the convergent 
 subsequence. For every subsequence $(\tilde{n}_k)_{k\in\N}\subset\N$, there 
 exist a further subsequence $(\tilde{n}_{k_l})_{l\in\N}$ and a strongly continuous 
 convex monotone semigroup $(\tilde{S}(t))_{t\geq 0}$ on $\Ck$ with 
 \[ \tilde{S}(t)f=\lim_{l\to\infty}I(\pi_{\tilde{n}_{k_l}}^t)f 
 	\quad\mbox{for all } (f,t)\in\Ck\times\T \]
 and the properties~(i)-(v). Hence, Theorem~\ref{thm:comp} implies
 \[ S(t)f=\tilde{S}(t)f \quad\mbox{for all } t\geq 0 \mbox{ and } f\in\Ck. \]
 Since every subsequence has a further subsequence which converges to a limit
 that is independent of the choice of the subsequence, we conclude that
 \[ S(t)f=\lim_{n\to\infty}I(\pi_n^t)f \quad\mbox{for all } (f,t)\in\Ck\times\T. \]
 In order to show that the convergence in the previous equation holds for arbitrary 
 points in time, let $t\geq 0$ and define $\tilde{\T}:=\T\cup\{t\}$. Analogously to the
 previous arguments, we obtain a semigroup $(\tilde{S}(t))_{t\geq 0}$ with
 \[ S(t)f=\tilde{S}(t)f=\lim_{n\to\infty}I(\pi_n^t)f 
	\quad\mbox{for all } (f,t)\in\Ck\times\tilde{\T}. \]
 Since $t\geq 0$ was arbitrary, we obtain equation~\eqref{eq:cher}. 
\end{proof}

In particular, the semigroup $(S(t))_{t\geq 0}$ is uniquely determined by the derivative 
$I'(0)f$ for smooth functions $f\in\Cb^\infty$.

\begin{corollary} \label{cor:IJ}
 Let $(I(t))_{t\geq 0}$ and $(J(t))_{t\geq 0}$ be two families of operators satisfying 
 Assumption~\ref{ass:cher} with $I'(0)f\leq J'(0)f$ for all $f\in\Cbi$. Then, 
 \[ S(t)f\leq T(t)f \quad\mbox{for all } t\geq 0 \mbox{ and } f\in\Ck, \]
 where $(S(t))_{t\geq 0}$ and $(T(t))_{t\geq 0}$ are the semigroups associated to 
 $(I(t))_{t\geq 0}$ and $(J(t))_{t\geq 0}$, respectively.
\end{corollary}
\begin{proof}
 This follows immediately from Theorem~\ref{thm:comp}. 
\end{proof}

\section{First order scaling limits}
\label{sec:LLN}

Throughout this section, we choose $\kappa\equiv 1$ and 
\[ \tilde{\kappa}\colon\R^d\to (0,\infty),\; x\mapsto (1+|x|)^{-1}. \]
Recall that the function $\tilde{\kappa}$ has been introduced in Remark~\ref{rem:cher}(b)
in order to verify Assumption~\ref{ass:cher}(vi) by means of a suitable moment condition.
In Appendix~\ref{app:E}, we gather some basic definitions and properties of convex expectations 
such as independence and weak convergence which resemble their classical analogues from probability 
theory and that are frequently used in the sequel.
Let $(\xi_n)_{n\in\N}\subset\H^d$ be an iid sequence of random vectors on a convex 
expectation space $(\Omega,\H,\bar{\E})$ with finite first moments and $(\psi_n)_{n\in\N}$
be a sequence of recursively defined functions $\psi_n\colon\R_+\times\R^d\times (\R^d)^n\to\R^d$. 
We define a sequence $(X_n)_{n\in\N}$ of random vectors by
\[ X_n:=\psi_n\big(\tfrac{1}{n},0,\xi_1,\ldots,\xi_n\big) \]
whose behaviour in the limit we are interested in. For 
$\psi_n(t,x,y_1,\ldots,y_n):=x+t\sum_{i=1}^n y_i$ we end up with an averaged sum of
iid samples, i.e., 
\[ X_n=\frac{1}{n}\sum_{i=1}^n \xi_i \quad\mbox{for all } n\in\N. \]
Since we are only interested in weak convergence, we mainly consider the distributions
of the vectors $X_n$. In this framework, requiring finite first moments means that the 
distribution of $\xi_1$ is well-defined even for continuous functions with at most linear 
growth at infinity, i.e., 
\[ F_{\xi_1}\colon {\rm C}_{\tilde{\kappa}}\to\R,\; f\mapsto\bar{\E}[f(\xi_1)]. \]
This functional is subsequently simply denoted by $\E[\,\cdot\,]$ and is supposed to be 
continuous from above in order to guarantee Assumption~\ref{ass:cher}(vi). We remark 
that, in the linear case, continuity from above is always given to due the fact that
$\lim_{c\to\infty}\e[|\xi_1|\one_{\{|\xi_1|\geq c\}}]=0$. Moreover, continuity from above on 
${\rm C}_{\tilde{\kappa}}$ is equivalent to the uniform integrability condition 
$\lim_{c\to\infty}\bar{\E}[(|\xi_1|-c)^+]=0$ from~\cite{Peng19}. We define $I(0):=\id_{\Cb}$ and
\[ (I(t)f)(x):=t\bar{\E}\left[\frac{1}{t}f\big(\psi_1(t,x,\xi_1)\big)\right] \]
for all $t>0$, $f\in\Cb$ and $x\in\R^d$. Then, the independence of $(\xi_n)_{n\in\N}$ implies 
\[ \big(I\big(\tfrac{1}{n}\big)^n f\big)(0)=\frac{1}{n}\bar{\E}\left[nf(X_n)\right]. \]
Hence, our operator theoretic result can be formulated equivalently by using classic
probabilistic notation. In Subsection~\ref{sec:LLN1},
we prove the previously discussed main result for first order scaling limits as well as
some immediate consequences. Motivated by Cram\'er's theorem, we then show that similar
convergence rates can be extended beyond the case of averaged sums of iid samples,
see Subsection~\ref{sec:bounds}. Moreover, we use a clever estimate from~\cite{Lacker}
to obtain polynomial convergence rates without requiring exponential moments. Finally, 
as an illustration of the previous rather abstract results, we consider convex expectations
that are defined as a (weighted) supremum over a set of probability measures, see
Subsection~\ref{sec:Wasser}.

\subsection{LLN for convex expectations and large deviations}
\label{sec:LLN1}

Let $\E\colon {\rm C}_{\tilde{\kappa}}\to\R$ be a convex expectation which is continuous from above, i.e., 
\begin{itemize}
 \item $\E[c]=c$ for all $c\in\R$, 
 \item $\E[f]\leq\E[g]$ for all $f,g\in {\rm C}_{\tilde{\kappa}}$ with $f\leq g$, 
 \item $\E[\lambda f+(1-\lambda)g]\leq\lambda\E[f]+(1-\lambda)\E[g]$ for all 
  $f,g\in {\rm C}_{\tilde{\kappa}}$ and $\lambda\in[0,1]$,
 \item $\E[f_n]\downarrow 0$ for all $(f_n)_{n\in\N}\subset {\rm C}_{\tilde{\kappa}}$ 
  with $f_n\downarrow 0$.
\end{itemize}
The convergence $f_n\downarrow 0$ is understood pointwise. However, it follows from Dini's
theorem and $\|f_n\|_{\tilde{\kappa}}\leq\|f_1\|_{\tilde{\kappa}}$ for all $n\in\N$ that 
$f_n\to 0$ in the mixed topology. By~\cite[Theorem~C.1]{BDKN}, there exists a convex monotone 
extension $\E\colon {\rm B}_{\tilde{\kappa}}\to\R$ such that, for every $\epsilon>0$ and $c\geq 0$, 
there exists $K\Subset\R^d$ with 
\begin{equation} \label{eq:tight}
 \E\left[\tfrac{c}{\kappa}\one_{K^c}\right]<\epsilon.
\end{equation}
Here, ${\rm B}_{\tilde{\kappa}}$ denotes the space of all Borel measurable functions $f\colon\R^d\to\R$ 
with $\|f\|_{\tilde{\kappa}}<\infty$. Furthermore, it holds that $\E[f_n]\downarrow\E[f]$ for all 
$(f_n)_{n\in\N}\subset {\rm C}_{\tilde{\kappa}}$ and $f\in {\rm C}_{\tilde{\kappa}}$ with $f_n\downarrow f$.

\begin{assumption} \label{ass:psi} 
 Let $\psi\colon\R_+\times\R^d\times\R^d\to\R^d$ be a continuous function which 
 satisfies the following conditions: 
 \begin{itemize}
  \item[(i)] There exists $L\geq 0$ such that, for all $t\in [0,1]$ and $x, y,z\in\R^d$, 
   \[ |x+\psi(t,y,z)-\psi(t,x+y,z)|\leq Lt|x|. \]
  Furthermore, it holds that $|\psi(t,x,y)-x|\leq L(1+|y|)t$ for all $t\in [0,1]$ and $x, y\in\R^d$.
  \item[(ii)] There exists a continuous function $\psi_0\colon\R^d\times\R^d\to\R^d$ with 
   \[ \lim_{h\downarrow 0}\sup_{x,y\in K}\left|\frac{\psi(h,x,y)-x}{h}-\psi_0(x,y)\right|=0
   	\quad\mbox{for all } K\Subset\R^d. \]
 \end{itemize}
\end{assumption}

The previous conditions are clearly satisfied for $\psi(t,x,y):=x+ty$ and
$\psi_0(x,y):=y$ in which case it holds 
$\psi_n(\frac{1}{n},x,y_1,\ldots,y_n)=x+\frac{1}{n}\sum_{i=1}^n y_i$. 
Another valid choice would be $\psi(t,x,y):=x+t\phi(x)+ty$ and $\psi_0(x,y):=\phi(x)+y$
for some Lipschitz continuous function $\phi\colon\R^d\to\R^d$ describing
a perturbation of the sample $\xi_{n+1}$ by a nonlinear function depending on 
the average of previous samples. In Section~\ref{sec:CLT}, the choice 
$\psi(t,x,y):=x+\sqrt{t}y$ leads to a CLT type result but here
Assumption~\ref{ass:psi}(i) is not satisfied since $|\psi(t,x,y)-x|=\sqrt{t}|y|$.
We define $I(0):=\id_{\Cb}$ and, for every $t>0$, $f\in\Cb$ and $x\in\R^d$, 
\[ (I(t)f)(x):=t\E\left[\frac{1}{t}f\big(\psi(t,x,\cdot)\big)\right] \]
Let $\Cb^1$ consist of all differentiable functions $f\in\Cb$ such that all partial derivatives 
are in $\Cb$ and denote by $xy:=\langle x,y\rangle$ the Euclidean inner product on $\R^d$. 
We want to point out that ${\rm C}_{\tilde{\kappa}}$ appears in this section only to ensure 
that $\E[\,\cdot\,]$ is defined and continuous from above on functions with at most linear 
growth at infinity. However, we consider $(I(t))_{t\geq 0}$ and $(S(t))_{t\geq 0}$ as operator 
families on $\Cb$. In particular, the definition of the (upper) Lipschitz sets and the properties 
from Theorem~\ref{thm:cher} are understood w.r.t. the supremum norm $\|\cdot\|_\infty$. Moreover,
the convergence in equation~\eqref{eq:LLN} and the definition of the generator $Af$ are understood 
as convergence in $\Cb$ w.r.t. the corresponding mixed topology. The sequence $(\psi_n)_{n\in\N}$ 
of functions $\psi_n\colon\R_+\times\R^d\times (\R^d)^n\to\R^d$ is recursively defined 
by $\psi_1:=\psi$ and
\[ \psi_{n+1}(t,x,y_1,\ldots,y_{n+1}):=\psi\big(t,\psi_n(t,x,y_1,\ldots,y_n),y_{n+1}\big). \]
The following theorem is the main result of this section.

\begin{theorem} \label{thm:LLN}
 Suppose that Assumption~\ref{ass:psi} holds. Then, there exists a strongly continuous 
 convex monotone semigroup $(S(t))_{t\geq 0}$ on $\Cb$ with
 \begin{equation} \label{eq:LLN}
  S(t)f=\lim_{n\to\infty}I\big(\tfrac{t}{n}\big)^n f\in\Cb
  \quad\mbox{for all } t\geq 0 \mbox{ and } f\in\Cb
 \end{equation}
 which has the properties~(i)-(v) from Theorem~\ref{thm:cher} with $\omega=0$ and 
 the constant $L\geq 0$ from Assumption~\ref{ass:psi}. For every $f\in\Cb^1$, it holds that 
 $f\in D(A)$ and 
 \[ (Af)(x)=\E[\nabla f(x)\psi_0(x,\cdot)] \quad\mbox{for all } x\in\R^d. \]
 In addition, for every convex expectation space $(\Omega,\H, \bar{\E})$ and iid sequence 
 $(\xi_n)_{n\in\N}\subset\H^d$ with $\bar{\E}[f(\xi_n)]=\E[f]$ for all $f\in\Cb$, 
 \begin{equation} \label{eq:LLN2}
  (S(1)f)(0)=\lim_{n\to\infty}\frac{1}{n}\bar{\E}\left[nf(X_n)\right]
  \quad\mbox{for all } f\in\Cb, 
 \end{equation}
 where $X_n:=\psi_n\big(\frac{1}{n},0,\xi_1,\ldots,\xi_n\big)$. 
\end{theorem}
\begin{proof}
 First, we verify Assumption~\ref{ass:cher}(i)-(iv), (vi) and~(vii). Lemma~\ref{lem:E}(ii) implies
 that Assumption~\ref{ass:cher}(i)-(iii) are satisfied with $\omega=0$. For every 
 $t\in (0,1]$, $r\geq 0$, $f\in\Lipb(r)$ and $x,y\in\R^d$, it follows from Lemma~\ref{lem:E}(ii)
 and Assumption~\ref{ass:psi}(i) that
 \begin{align}
  |I(t)(\tau_x f)-\tau_x I(t)f|(x)
  &=t\left|\E\left[\frac{1}{t}f(x+\psi(t,y,\cdot))\right]-\E\left[\frac{1}{t}f(\psi(t,x+y,\cdot))\right]\right| 
   \nonumber \\
  &\leq\|f(x+\psi(t,y,\cdot))-f(\psi(t,x+y,\cdot))\|_\infty \nonumber \\
  &\leq r\|x+\psi(t,y,\cdot)-\psi(t,x+y,\cdot)\|_\infty\leq Lrt|x|. \label{eq:shift}
 \end{align}
 Thus, Assumption~\ref{ass:cher}(iv) is satisfied. For every $f\in\Cb$ with 
 $\|f\|_{\tilde{\kappa}}\leq 1$, $t\in (0,1]$ and $x\in\R^d$, we use Lemma~\ref{lem:E}(iii) and~(v) and 
 Assumption~\ref{ass:psi}(i) to estimate
 \begin{align*}
  |(I(t)f)(x)|
  &\leq t\E\left[\frac{1}{t}|f(\psi(t,x,\cdot))|\frac{{\tilde{\kappa}}(\psi(t,x,\cdot))}{{\tilde{\kappa}}(\psi(t,x,\cdot))}\right] 
  \leq t\|f\|_{\tilde{\kappa}}\E\left[\frac{1}{t}\big(1+|\psi(t,x,\cdot)|\big)\right] \\
  &\leq\|f\|_{\tilde{\kappa}}\big(1+|x|+t\E[\tfrac{L}{{\tilde{\kappa}}}\big|\leq e^{ct}\|f\|_{\tilde{\kappa}}(1+|x|),
 \end{align*}
 where $c:=\E[\nicefrac{L}{{\tilde{\kappa}}}]$. This shows that $\|I(t)f\|_{\tilde{\kappa}}\leq e^{ct}\|f\|_{\tilde{\kappa}}$ 
 and Remark~\ref{rem:cher}(b) yields that Assumption~\ref{ass:cher}(vi) is satisfied. 
 Concerning Assumption~\ref{ass:cher}(vii), we remark that inequality~\eqref{eq:shift} implies 
 \begin{align*}
  |(I(t)f)(x+y)-(I(t)f)(x)|
  &\leq\|\tau_y I(t)f-I(t)(\tau_y f)\|_\infty+\|I(t)(\tau_y f)-I(t)f\|_\infty \\
  &\leq Lrt|y|+\|\tau_y f-f\|_\infty\leq e^{Lt}r|y|
 \end{align*}
 for all $r, t\geq 0$, $f\in\Lipb(r)$ and $x,y\in\R^d$ and therefore 
 $I(t)\colon\Lipb(r)\to\Lipb(e^{Lt}r)$. Moreover, for every $r\geq 0$ and $f\in\Lipb(r)$, 
 it follows from Assumption~\ref{ass:psi}(i) that 
 \[  |f(\psi(t,x,y))-f(x)|\leq r|\psi(t,x,y)-x|\leq Lr(1+|y|)t \]
 for all $r,t\geq 0$, $f\in\Lipb(r)$ and $x,y\in\R^d$. We use Lemma~\ref{lem:E}(v) to obtain
 \[ |(I(t)f-f)(x)|\leq t\E\left[\frac{|f(\psi(t,x,\cdot)-f(x)|}{t}\right]
 	\leq\E[\tfrac{Lr}{\kappa}]t \]
 and therefore $\Lipb\subset\L^I$

 Second, we show that $(I'(0)f)(x)=\E[\nabla f(x)\psi_0(x,\cdot)]$ for all 
 $f\in\Cb^1$ and $x\in\R^d$. For every $h>0$, $f\in\Cb^1$, $x\in\R^d$ and 
 $\lambda\in (0,1]$, we use Lemma~\ref{lem:lambda} to estimate 
 \begin{align*}
  &\left(\frac{I(h)f-f}{h}\right)(x)-\E[\nabla f(x)\psi_0(x,\cdot)] \\
  &=\E\left[\left(\frac{\psi(h,x,\cdot)-x}{h}\right)\int_0^1 \nabla f(x_s)\,\d s\right]
  -\E[\nabla f(x)\psi_0(x,\cdot)] \\
  &\leq\lambda\E\left[\frac{1}{\lambda}\left(\left(\frac{\psi(h,x,\cdot)-x}{h}\right)
  \int_0^1 \nabla f(x_s)\,\d s-\nabla f(x)\psi_0(x,\cdot)\right)
  +\nabla f(x)\psi_0(x,\cdot)\right] \\
  &\quad\; -\lambda\E[\nabla f(x)\psi_0(x,\cdot)], 
 \end{align*}
 where $x_s:=x+s(\psi(t,x,\cdot)-x)$. Furthermore, the convexity of $\E$ implies
 \begin{align*}
  &\lambda\E\left[\frac{1}{\lambda}\left(\left(\frac{\psi(h,x,\cdot)-x}{h}\right)
  \int_0^1 \nabla f(x_s)\,\d s-\nabla f(x)\psi_0(x,\cdot)\right)
  +\nabla f(x)\psi_0(x,\cdot)\right] \\
  &\leq\frac{\lambda}{4}\E\left[\frac{4}{\lambda}\left(\left(\frac{\psi(h,x,\cdot)-x}{h}\right)
  \int_0^1 (\nabla f(x_s)-\nabla f(x))\,\d s\right)\one_{B(n)}(\cdot)\right] \\
  &\quad\; +\frac{\lambda}{4}\E\left[\frac{4}{\lambda}\left(\left(\frac{\psi(h,x,\cdot)-x}{h}\right)
  \int_0^1 (\nabla f(x_s)-\nabla f(x))\,\d s\right)\one_{B(n)^c}(\cdot)\right] \\
  &\quad\; +\frac{\lambda}{4}\E\left[\frac{4}{\lambda}\left(\frac{\psi(h,x,\cdot)-x}{h}-\psi_0(x,\cdot)\right)
  \nabla f(x)\right] 
  +\frac{\lambda}{4}\E[4\nabla f(x)\psi_0(x,\cdot)]
 \end{align*}
 for all $n\in\N$ and $B(n):=B_{\R^d}(n)$. Let $\epsilon>0$ and $K\Subset\R^d$. 
 By Assumption~\ref{ass:psi}, there exists $\lambda\in (0,1]$ with 
 \[ \sup_{x\in\R^d}\lambda\E[4\|\nabla f\|_\infty|\psi_0(x,\cdot)|]
 	\leq\sup_{x\in\R^d}\lambda\E\left[\frac{4L\|\nabla f\|_\infty}{\kappa}\right]
 	\leq\frac{\epsilon}{2}.\]
 Furthermore, due to Assumption~\ref{ass:psi} and inequality~\eqref{eq:tight}, there exist 
 $K_1\Subset\R^d$ and $h_0\in (0,1]$ with
 \begin{align*}
  &\frac{\lambda}{4}\E\left[\frac{4}{\lambda}\left(\frac{\psi(h,x,\cdot)-x}{h}-\psi_0(x,\cdot)\right)
  \nabla f(x)\right] \\
  &\leq\frac{\lambda}{4}\E\left[\frac{4}{\lambda}\left|\frac{\psi(h,x,\cdot)-x}{h}-\psi_0(x,\cdot)\right|
  \cdot\|\nabla f\|_\infty\right] \\
  &\leq\frac{\lambda}{8}\E\left[\frac{8\|\nabla f\|_\infty}{\lambda}
  \left|\frac{\psi(h,x,\cdot)-x}{h}-\psi_0(x,\cdot)\right|\one_{K_1}(\cdot)\right]
  +\frac{\lambda}{8}\E\left[\frac{16L\|\nabla f\|_\infty}{\lambda\kappa}\one_{K_1^c}(\cdot)\right]
  \leq\frac{\epsilon}{8}
 \end{align*}
 for all $x\in K$ and $h\in (0,h_0]$. By Assumption~\ref{ass:psi}(i) and inequality~\eqref{eq:tight}, 
 there exist $n\in\N$ with
 \begin{align*}
  &\frac{\lambda}{4}\E\left[\frac{4}{\lambda}\left(\left(\frac{\psi(h,x,\cdot)-x}{h}\right)
  \int_0^1 (\nabla f(x_s)-\nabla f(x))\,\d s\right)\one_{B(n)^c}(\cdot)\right] \\
  &\leq\frac{\lambda}{4}\E\left[\frac{8L\|\nabla f\|_\infty}{\lambda\kappa}\one_{B(n)^c}(\cdot)\right]
  \leq\frac{\epsilon}{8}.
 \end{align*}
 for all $x\in\R^d$ and $h\in (0,h_0]$. Since $K\Subset\R^d$ is compact, there exists 
 $\delta>0$ with 
 \[ |\nabla f(x+y)-\nabla f(x)|<\frac{\epsilon}{8Ln}
 	\quad\mbox{for all } x\in K \mbox{ and } y\in B(\delta). \]
 Hence, by Assumption~\ref{ass:psi}(i), there exist $h_1\in (0,h_0]$ with 
 \begin{align*}
  \frac{\lambda}{4}\E\left[\frac{4}{\lambda}\left(\left(\frac{\psi(h,x,\cdot)-x}{h}\right)
  \int_0^1 (\nabla f(x_s)-\nabla f(x))\,\d s\right)\one_{B(n)}(\cdot)\right]
  \leq\frac{\lambda}{4}\E\left[\frac{4Ln}{\lambda}\frac{\epsilon}{8Ln}\right]
  =\frac{\epsilon}{8}
 \end{align*}
 for all $x\in K$ and $h\in (0,h_1]$. It follows from the previous estimates that
 \[ \left(\frac{I(h)f-f}{h}\right)(x)-\E[\nabla f(x)\psi_0(x,\cdot)]\leq\epsilon \]
 for all $x\in K$ and $h\in (0,h_1]$. Concerning the lower bound, Lemma~\ref{lem:lambda}
 yields
 \begin{align*}
  &\E\left[\left(\frac{\psi(h,x,\cdot)-x}{h}\right)\int_0^1 \nabla f(x_s)\,\d s\right]
  -\E[\nabla f(x)\psi_0(x,\cdot)] \\
  &\geq -\lambda\E\left[\frac{\nabla f(x)\psi_0(x,\cdot)-y_s}{\lambda}+y_s\right]
  +\lambda\E[y_s]
 \end{align*}
 for all $h>0$, $x\in\R^d$ and $\lambda\in (0,1]$, where 
 \[ y_s:=\left(\frac{\psi(h,x,\cdot)-x}{h}\right)\int_0^1 \nabla f(x_s)\,\d s. \]
 Furthermore, we use the convexity of $\E$ to estimate
 \begin{align*}
  &\lambda\E\left[\frac{\nabla f(x)\psi_0(x,\cdot)-y_s}{\lambda}+y_s\right] \\
  &\leq\frac{\lambda}{4}\E\left[\frac{4}{\lambda}
  \left(\psi_0(x,\cdot)-\frac{\psi(h,x,\cdot)-x}{h}\right)\int_0^1 \nabla f(x_s)\,\d s\right] \\
  &\quad\; +\frac{\lambda}{4}\E\left[\frac{4}{\lambda}
  \left(\psi_0(x,\cdot)\int_0^1 \big(\nabla f(x)-\nabla f(x_s)\big)\,\d s\right)\one_{B(n)}(\cdot)\right] \\
  &\quad\; +\frac{\lambda}{4}\E\left[\frac{4}{\lambda}
  \left(\psi_0(x,\cdot)\int_0^1 \big(\nabla f(x)-\nabla f(x_s)\big)\,\d s\right)\one_{B(n)^c}(\cdot)\right] \\
  &\quad\; +\frac{\lambda}{4}\E\left[4\left(\frac{\psi(h,x,\cdot)-x}{h}\right)
  \int_0^1 \nabla f(x_s)\,\d s\right].
 \end{align*}
 for all $n\in\N$. Analogously to the upper bound, one can estimate the terms on 
 the right-hand side of the previous inequality to obtain
 \[ \left(\frac{I(h)f-f}{h}\right)(x)-\E[\nabla f(x)\psi_0(x,\cdot)]\geq-\epsilon \]
 for all $x\in K$ and sufficiently small $h>0$. This shows 
 \[ \lim_{h\downarrow 0}\sup_{x\in K}
 	\left|\left(\frac{I(h)f-f}{h}\right)(x)-\E[\nabla f(x)\psi_0(x,\cdot)]\right|=0 
 	\quad\mbox{for all } K\Subset\R^d. \]
 We obtain that $I'(0)f\in\Cb$ exists and is given by 
 \[ (I'(0)f)(x)=\E[\nabla f(x)\psi_0(x,\cdot)] \quad\mbox{for all } x\in\R^d. \]
 Now, the first part of the claim follows from Theorem~\ref{thm:cher}. 
 
 Third, we verify equation~\eqref{eq:LLN2}. Choose $t:=1$ and $h_n:=\nicefrac{1}{n}$ 
 for all $n\in\N$. Due to Corollary~\ref{cor:IJ}, the limit in equation~\eqref{eq:LLN}
 does not depend on the choice of the partition. In particular, for every $f\in\Cb$ 
 and $x\in\R^d$, 
 \[ (S(1)f)(x)=\lim_{n\to\infty}(I(\tfrac{1}{n})^n f)(x). \]
 For fixed $n\in\N$ and $h:=\nicefrac{1}{n}$, we show by induction that
 \[ (I(h)^k f)(x)=h\bar{\E}\left[\frac{1}{h}f\big(\psi_k(h,x,\xi_1,\ldots,\xi_k)\big)\right] \]
 for all $f\in\Cb$, $x\in\R^d$ and $k\in\N$. For $k=1$, we have 
 \[ (I(h)f)(x)=h\E\left[\frac{1}{h}f\big(\psi(h,x,\cdot)\big)\right]
 	=h\bar{\E}\left[\frac{1}{h}f\big(\psi_1(h,x,\xi_1)\big)\right]. \]
 For the induction step, we use that $\xi_{k+1}$ is independent of $(\xi_1,\ldots,\xi_k)$
 and has the same distribution as $\xi_1$ to obtain
 \begin{align*}
  (I(h)^{k+1} f)(x) &=(I(h)^k I(h)f)(x)
  =h\bar{\E}\left[\frac{1}{h}(I(h)f)\big(\psi_k(h,x,\xi_1,\ldots,\xi_k)\big)\right] \\
  &=h\bar{\E}\left[\bar{\E}\left[\frac{1}{h}f\big(\psi(h,\psi_k(h,x,y_1,\ldots,y_k),\xi_{k+1})\big)\right]
  \Big|_{(y_1,\ldots,y_k)=(\xi_1,\ldots,\xi_k)}\right] \\
  &=h\bar{\E}\left[\frac{1}{h}f\big(\psi(h,\psi_k(h,x,\xi_1,\ldots,\xi_k),\xi_{k+1})\big)\right] \\
  &=h\bar{\E}\left[\frac{1}{h}f\big(\psi_{k+1}(h,x,\xi_1,\ldots,\xi_k,\xi_{k+1})\big)\right]. \qedhere
 \end{align*}
\end{proof}

The existence of $(\Omega,\H,\bar{\E})$ and $(\xi_n)_{n\in\N}\subset\H^d$
as in the previous theorem is always guaranteed by a nonlinear version of Kolmogorov's 
extension theorem, see Theorem~\ref{thm:kol}. For $\psi(t,x,y):=x+ty$ and a
sublinear expectation $\E[\,\cdot\,]$, one can estimate 
\[ \left(\frac{I(h)f-f}{h}\right)(x)-\E[\nabla f(x)\xi] 
    \leq\E\left[\int_0^1 |\nabla f(x_s)-\nabla f(x)|\cdot |\xi|\,\d s\right] \]
with $\xi:=\id_{\R^d}$ and the computation of the generator simplifies accordingly.
However, the main ideas of the proof and the result remain the same. To give an explicit 
example for the sequence $(\psi_n)_{n\in\N}$, we consider the case that the sample $\xi_{n+1}$ 
is randomly shifted by a nonlinear function depending on the average of the previous samples 
$\xi_1,\ldots,\xi_n$.

\begin{corollary} \label{cor:LLN}
 Let $\psi(t,x,y):=x+\phi(t,x)+ty$ for a function $\phi\colon\R_+\times\R^d\to\R^d$
 such that there exist $L\geq 0$ with $\phi(t,\cdot)\in\Lipb(Lt)$ for all $t\geq 0$. 
 Furthermore, we assume that the limit 
 $\phi_0:=\lim_{h\downarrow 0}\nicefrac{\phi(h,\cdot)}{h}\in\Cb$ exists. Then, 
 there exists a strongly continuous convex monotone semigroup $(S(t))_{t\geq 0}$ 
 on $\Cb$ with
 \[ S(t)f=\lim_{n\to\infty}I\big(\tfrac{t}{n}\big)^n f\in\Cb
 	\quad\mbox{for all } t\geq 0 \mbox{ and } f\in\Cb \]
 which has the properties~(i)-(v) from Theorem~\ref{thm:cher} with $\omega=0$ and 
 the constant $L\geq 0$ from above. For every $f\in\Cb^1$ and $x\in\R^d$, it holds that 
 $f\in D(A)$ and 
 \[ (Af)(x)=\E[\nabla f(x)\xi]+\phi_0(x)\nabla f(x),
 	\quad\mbox{where } \xi:=\id_{\R^d}. \]
 In addition, for every convex expectation space $(\Omega,\H, \bar{\E})$ and iid sequence 
 $(\xi_n)_{n\in\N}\subset\H^d$ with $\bar{\E}[f(\xi_n)]=\E[f]$ for all $f\in\Cb$, 
 \begin{equation} 
  (S(1)f)(0)=\lim_{n\to\infty}\frac{1}{n}\bar{\E}\left[nf(X_n)\right]
  \quad\mbox{for all } f\in\Cb, 
 \end{equation}
 where $X_n:=\psi_n(\frac{1}{n},0,\xi_1,\ldots,\xi_n)$. 
\end{corollary} 
\begin{proof}
 It is straightforward to show that Assumption~\ref{ass:psi} is satisfied with
 \[ \psi_0(x,y):=\phi_0(x)+y \quad\mbox{for all } x,y\in\R^d. \]
 Hence, the claim follows from Theorem~\ref{thm:LLN}. 
\end{proof}

For non-perturbed averaged sums of iid samples, the semigroup can be represented 
explicitly by the Hopf--Lax formula.

\begin{theorem} \label{thm:LLN2}
 Let $\psi(t,x,y):=x+ty$ for all $t\geq 0$ and $x,y\in\R^d$. Then, the semigroup 
 $(S(t))_{t\geq 0}$ from Theorem~\ref{thm:LLN} has the representation
 \[ (S(t)f)(x)=\sup_{y\in\R^d}\big(f(x+ty)-\phi(y)t\big)
 	\quad\mbox{for all } t\geq 0, f\in\Cb \mbox{ and } x\in\R^d, \]
 where $\phi(y):=\sup_{z\in\R^d}(yz-\E[z\xi])$ and $\xi:=\id_{\R^d}$. 
 For every $f\in\Cb^1$, it holds that
 \[ (Af)(x)=\E[\nabla f(x)\xi] \quad\mbox{for all } x\in\R^d. \]
 In addition, for every convex expectation space $(\Omega,\H, \bar{\E})$ and iid sequence 
 $(\xi_n)_{n\in\N}\subset\H^d$ with $\bar{\E}[f(\xi_n)]=\E[f]$ for all $f\in\Cb$, 
 \[ \lim_{n\to\infty}\frac{1}{n}\bar{\E}\left[nf\left(\frac{1}{n}\sum_{i=1}^n \xi_i\right)\right]
 	=\sup_{x\in\R^d}\big(f(x)-\phi(x)\big) \quad\mbox{for all } f\in\Cb. \]
\end{theorem}
\begin{proof}
 For every $t\geq 0$, $f\in\Cb$ and $x\in\R^d$, we define
 \[ (T(t)f)(x):=\sup_{y\in\R^d}\big(f(x+ty)-\phi(y)t\big). \]
 We show that $(T(t))_{t\geq 0}$ is a semigroup satisfying Assumption~\ref{ass:cher}.
 The convexity of the mapping $\R^d\to\R,\; x\mapsto\E[x\xi]$ and Fenchel--Moreau's 
 theorem yield
 \[ \E[x\xi]=\sup_{y\in\R^d}\big(xy-\phi(y)\big) \quad\mbox{for all } x\in\R^d.\]
 In particular, we obtain $\inf_{x\in\R^d}\phi(x)=-\E[0]=0$. 
 For every $t\geq 0$ and $f\in\Cb$, it follows from $\lim_{|x|\to\infty}\phi(x)=\infty$
 that there exists $K\Subset\R^d$ with
 \[ (T(t)f)(x)=\sup_{y\in K}\big(f(x+ty)-\phi(y)t\big) \quad\mbox{for all } x\in\R^d. \]
 Since $f$ is uniformly continuous on compacts, we obtain $T(t)f\in\Cb$. Clearly, 
 Assumption~\ref{ass:cher}(i)-(iv) and~(vii) are satisfied with $\omega=L=0$, where
 $T(t)0=0$ follows from $\inf_{x\in\R^d}\phi(x)=0$. For every $r\geq 0$ and $f\in\Lipb(r)$,
 we use
 \[ f(x+ty)-\phi(y)t\leq f(x)+\big(r|y|-\phi(y)\big)t \]
 and $\lim_{|x|\to\infty}\nicefrac{\phi(x)}{|x|}=\infty$ to choose $K\Subset\R^d$ with 
 \begin{equation} \label{eq:cramer}
  (T(t)f)(x)=\sup_{y\in K}\big(f(x+ty)-\phi(y)t\big) 
  \quad\mbox{for all } t\geq 0 \mbox{ and } x\in\R^d. 
 \end{equation}
 We obtain $0\leq T(t)f-f\leq\sup_{y\in K}rt|y|$ for all $t\geq 0$ which shows 
 $\Lipb\subset\L^T$. Let $f\in\Cb^1$ and choose $K\Subset\R^d$ such that 
 equation~\eqref{eq:cramer} holds. For every $x\in\R^d$, we use Fenchel--Moreau's 
 theorem to estimate
 \begin{align*}
  \left|\frac{(T(t)f-f)(x)}{t}-\E[\nabla f(x)\xi]\right|
  &\leq\sup_{y\in K}\left|\frac{f(x+ty)-f(x)}{t}-\nabla f(x)y\right| \\
  &\leq\sup_{y\in K}\frac{1}{t}\int_0^t |\nabla f(x+sy)-\nabla f(x)|\cdot|y|\,\d s.
 \end{align*}
 Since $\nabla f$ is uniformly continuous on compacts, we obtain $f\in D(B)$ 
 and $Af=Bf$, where $B$ denotes the generator of $(T(t))_{t\geq 0}$. In particular,
 Assumption~\ref{ass:cher}(v) is satisfied. Next, we show that $(T(t))_{t\geq 0}$
 forms a semigroup. For every $s,t\geq 0$, $f\in\Cb$ and $x\in\R^d$, 
 \begin{align*}
  (T(s+t)f)(x)
  &=\sup_{y\in\R^d}\big(f(x+sy+ty)-\phi(y)s-\phi(y)t\big) \\
  &\leq\sup_{y\in\R^d}\sup_{z\in\R^d}\big(f(x+sy+tz)-\phi(y)s-\phi(z)t\big)
  =(T(s)T(t)f)(x).
 \end{align*}
 Furthermore, due to $\lim_{|x|\to\infty}\phi(x)=\infty$, there exists $y_s, y_t\in\R^d$
 with 
 \[ (T(s)T(t)f)(x)=(T(t)f)(x+sy_s)-\phi(y_s)s=f(x+sy_s+ty_t)-\phi(y_s)s-\phi(y_t)t. \]
 For $y_{s+t}:=\frac{s}{s+t}y_s+\frac{t}{s+t}y_t$, it follows from the convexity of
 $\phi$ that
 \[ \phi(y_{s+t})(s+t)=\phi\left(\frac{s}{s+t}y_s+\frac{t}{s+t}y_t\right)(s+t)
 	\leq\phi(y_s)s+t\phi(y_t)t. \]
 We obtain 
 \begin{align*}
  (T(s)T(t)f)(x) &=f(x+(s+t)y_{s+t})-\phi(y_s)s-\phi(y_t)t \\
  &\leq f(x+(s+t)y_{s+t})-\phi(y_{s+t})(s+t)\leq (T(s+t)f)(x). 
 \end{align*}
 It remains to verify Assumption~\ref{ass:cher}(vi). For every $t\geq 0$, $x\in\R^d$
 and $(f_n)_{n\in\N}\subset\Cb$ with $f_n\downarrow 0$, there exists $K\Subset\R^d$ 
 with
 \[ (T(t)f_n)(x)=\sup_{y\in K}\big(f_n(x+ty)-\phi(y)t\big)
 	\quad\mbox{for all } n\in\N \]
 such that Dini's theorem implies $(T(t)f_n)(x)\downarrow 0$ as $n\to\infty$. 
 Since the mapping 
 \[ \R_+\times\R^d\to\R,\; (t,x)\mapsto (T(t)f_n)(x) \]
 is continuous for all $n\in\N$, we can use Dini's theorem again to obtain
 \[ \sup_{(s,x)\in [0,t]\times K}(T(s)f)(x)\downarrow 0 
 	\quad\mbox{for all } t\geq 0 \mbox{ and } K\Subset\R^d. \]
 The previous statement remains valid for sequences $(f_n)_{n\in\N}\subset\Cb$
 because every function in $\Cb$ can be approximated from above by a decreasing
 sequence in $\Lipb$. Finally, we can apply Corollary~\ref{cor:IJ} to obtain $S(t)f=T(t)f$ 
 for all $t\geq 0$ and $f\in\Cb$. Furthermore, 
 \[ \psi_n(t,x,y_1,\ldots,y_n)=x+t\sum_{i=1}^n y_i \]
 for all $n\in\N$, $t\geq 0$ and $x,y_1,\ldots,y_n\in\R^d$. Thus, the second part of
 the claim follows from Theorem~\ref{thm:LLN}. 
\end{proof}

The previous theorem extends Peng's results from the sublinear to the convex case. 
Indeed, if $\E[\,\cdot\,]$ is a sublinear expectation, it holds that
\[ \lim_{n\to\infty}\bar{\E}\left[f\left(\frac{1}{n}\sum_{i=1}^n \xi_i\right)\right]
	=\sup_{\{x\in\R^d\colon\phi(x)=0\}}f(x) \quad\mbox{for all } f\in\Cb. \]
Hence, we obtain convergence to a maximal distribution as in~\cite{Peng08b, Peng19} 
and the limit in Theorem~\ref{thm:LLN2} can be seen as a convex version thereof. 
Furthermore, as an immediate consequence of the previous result, we obtain Cram\'er's 
theorem as LLN for the entropic risk measure. Let $(\Omega,\F,\P)$ be a probability 
space and $(\xi_n)_{n\in\N}$ be an iid sequence of random vectors 
$\xi_n\colon\Omega\to\R^d$ with $\e[e^{c|\xi_1|}]<\infty$ for all $c\geq 0$. Let
\[ \Lambda\colon\R^d\to\R,\; x\mapsto\log\big(\e[e^{x\xi_1}]\big) \]
be the logarithmic moment generating function and denote by
\[ \Lambda^*\colon\R^d\to\R,\; x\mapsto\sup_{y\in\R^d}\big(xy-\Lambda(y)\big) \]
its convex conjugate. Moreover, we define $X_n:=\frac{1}{n}\sum_{i=1}^n \xi_i$ for all $n\in\N$.

\begin{corollary}[Cram\'er] \label{cor:cramer}
 For every $f\in\Cb$, 
 \[ \lim_{n\to\infty}\frac{1}{n}\log\big(\e\big[e^{nf(X_n)}\big]\big)
 	=\sup_{x\in\R^d}\big(f(x)-\Lambda^*(x)\big). \]
\end{corollary}
\begin{proof}
 Clearly, the functional $\E\colon {\rm C}_{\tilde{\kappa}}\to\R,\; f\mapsto\log(\e[e^{f(\xi_1)}])$ 
 is convex and monotone with $\E[c]=c$ for all $c\in\R$. In addition, the dominated convergence 
 theorem implies $\lim_{n\to\infty}\E[|\cdot|\one_{B(n)^c}]=0$ which shows that $\E$ 
 is continuous from above. Hence, we can apply Theorem~\ref{thm:LLN2} to obtain the claim. 
\end{proof}

\subsection{Upper bounds and convergence rates}
\label{sec:bounds}

Let $\E\colon {\rm C}_{\tilde{\kappa}}\to\R$ be a convex expectation which is continuous from above
and $\psi\colon\R_+\times\R^d\times\R^d\to\R^d$ be a function satisfying 
Assumption~\ref{ass:psi}. For $\psi(t,x,y):=x+ty$, the semigroup that we obtain in the 
limit can explicitly be represented by the Hopf--Lax formula, see Theorem~\ref{thm:LLN2}.
In general, such a formula is not available but we can still provide explicit (upper) bounds 
and resulting convergence rates. Let $H_\pm\colon\R^d\to\R$ be convex functions with 
\[ H_-(y)\leq\E[\psi_0(x,\cdot)y]\leq H_+(y)
	\quad\mbox{for all } x,y\in\R^d \]
and $H_\pm(0)=0$. Note that at least an upper bound always exists, since we can 
choose 
\[ H_+(y):=\sup_{x\in\R^d}\E[\psi_0(x,\cdot)y]\leq\E[L(1+|\cdot|)y]. \]
Moreover, under the assumptions of Corollary~\ref{cor:LLN} we can choose 
\[ H_+(y):=\E[y\xi]+\sup_{x\in\R^d}\phi_0(x)y\leq\E[y\xi]+L|y|, 
	\quad\mbox{where } \xi:=\id_{\R^d}. \]
For the following lemma, let $(S(t))_{t\geq 0}$ be the semigroup from Theorem~\ref{thm:LLN} 
given by
\[ S(t)f=\lim_{n\to\infty}I(\pi_n^t)f \quad\mbox{for all } t\geq 0 \mbox{ and } f\in\Cb, \]
where $I(0):=\id_{\Cb}$ and, for every $t>0$, $f\in\Cb$ and $x\in\R^d$, 
\[ (I(t)f)(x):=t\E\left[\frac{1}{t}f(\psi(t,x,\cdot))\right]. \]

\begin{lemma} \label{lem:pm}
 It holds that $S_-(t)f\leq S(t)f\leq S_+(t)f$ for all $t\geq 0$ and $f\in\Cb$, where
 \[ (S_\pm(t)f)(x):=\sup_{y\in\R^d}\big(f(x+ty)-H_\pm^*(y)t\big) \]
 and $H_\pm^*(y):=\sup_{z\in\R^d}(yz-H_\pm(z))$. 
\end{lemma}
\begin{proof}
 As seen during the proof of Theorem~\ref{thm:LLN2}, the families $(S_\pm(t))_{t\geq 0}$
 are semigroups satisfying Assumption~\ref{ass:cher} with generators
 \[ (A_\pm f)(x)=H_\pm(\nabla f(x)) 
 	\quad\mbox{for all } f\in\Cb^1 \mbox{ and } x\in\R^d. \]
 In particular, it holds that $A_- f\leq Af\leq A_+ f$ for all $f\in\Cb^1$ and Corollary~\ref{cor:IJ} 
 implies
 \[ S_-(t)f\leq S(t)f\leq S_+(t)f
 	\quad\mbox{for all } t\geq 0 \mbox{ and } f\in\Cb. \qedhere \]
\end{proof}

We illustrate the previous estimate for the entropic risk measure which yields exponential 
convergence rates in the situation of Corollary~\ref{cor:LLN}, where we considered averaged
sums of perturbed iid samples. Recall that the perturbation of the $\xi_{n+1}$ consisted 
of a random shift by a nonlinear depending on the average of the previous samples $\xi_1,\ldots,\xi_n$. 
It turns out that the well-known convergence rate from the case of unperturbed iid samples 
is reduced according to the size of the shift, see Theorem~\ref{thm:cramer}. Furthermore, 
if we only require that $(\xi_n)_{n\in\N}$ has finite $p$-th moments instead of finite 
exponential moments, we still obtain polynomial convergence rates as in~\cite{Lacker},
see Theorem~\ref{thm:cramer2}. For the following two theorems, let $(\Omega,\F,\P)$ 
be a probability space and $(\xi_n)_{n\in\N}$ be an iid sequence of random vectors 
$\xi_n\colon\Omega\to\R^d$. Furthermore, let $\phi\colon\R_+\times\R^d\to\R^d$ be 
a function such that there exist $L\geq 0$ with $\phi(t,\cdot)\in\Lipb(Lt)$ for all $t\geq 0$
and suppose that the limit $\phi_0:=\lim_{h\downarrow 0}\nicefrac{\phi(h,\cdot)}{h}\in\Cb$ 
exists. We define $\psi(t,x,y):=x+\phi(t,x)+ty$ and recursively a sequence $(\psi_n)_{n\in\N}$ 
of functions $\psi_n\colon\R_+\times\R^d\times (\R^d)^n\to\R^d$ by $\psi_1:=\psi$ and
\[ \psi_{n+1}(t,x,y_1,\ldots,y_{n+1}):=\psi\big(t,\psi_n(t,x,y_1,\ldots,y_n),y_{n+1}\big). \]
Finally, let $X_n:=\psi_n(\frac{1}{n}, 0, \xi_1,\ldots,\xi_n)$ for all $n\in\N$.

\begin{theorem} \label{thm:cramer}
 Assume that $\e[e^{c|\xi_1|}]<\infty$ for all $c\geq 0$. Denote by
 \[ \Lambda\colon\R^d\to\R,\; x\mapsto\log\big(\e[e^{x\xi_1}]\big) \]
 the logarithmic moment generating function and by
 \[ \Lambda^*\colon\R^d\to\R,\; x\mapsto\sup_{y\in\R^d}\big(xy-\Lambda(y)\big) \]
 its convex conjugate. Then, 
 \[ \lim_{n\to\infty}\frac{1}{n}\log\big(\e[e^{nf(X_n)}]\big)
 	 \leq\sup_{x\in\R^d}\big(f(x)-H_+^*(x)\big)
 	 \quad\mbox{for all } f\in\Cb, \]
 where $H_+^*(x):=\inf_{y\in B(L)}\Lambda^*(x+y)$. Furthermore,
 \[ \limsup_{n\to\infty}\frac{1}{n}\log(\P(X_n\in A))\leq -\inf_{x\in A_L}\Lambda^*(x) \]
 for all closed sets $A\subset\R^d$, where $A_L:=\{x+y\colon x\in A, \, y\in B(L)\}$. 
\end{theorem}
\begin{proof}
 Applying Corollary~\ref{cor:LLN} with $\E[f]:=\log(\e[e^{f(\xi_1)}])$ yields a semigroup 
 $(S(t))_{t\geq 0}$ with generator $(Af)(x)=\Lambda(\nabla f(x))+\phi_0(x)\nabla f(x)$
 for all $f\in\Cb^1$ and $x\in\R^d$. Moreover, for every $t\geq 0$, $f\in\Cb$ and $x\in\R^d$, 
 we define $H_+(x):=\Lambda(x)+L|x|$ and 
 \[ (S_+(t)f)(x):=\sup_{y\in\R^d}\big(f(x+ty)-H_+^*(y)t\big), \]
 where $H_+^*(y):=\sup_{z\in\R^d}(yz-H_+(z))$. Lemma~\ref{lem:pm} implies $S(t)f\leq S_+(t)f$ 
 for all $t\geq 0$ and $f\in\Cb$. Moreover, by Fenchel--Moreau's theorem, the function
 \[ f\colon\R^d\to\R,\; x\mapsto\inf_{x=y+z}\big(\Lambda^*(y)+\infty\one_{B(L)^c}(z)\big) \]
 satisfies $f^*(x)=\Lambda(x)+L|x|=H_+(x)$ and thus
 \[ H_+^*(x)=f(x)=\inf_{y\in B(L)}\Lambda^*(x+y). \]
 Let $f:=-\infty\one_{A^c}\in\Ub$ for a closed subset $A\subset\R^d$ and 
 $(f_n)_{n\in\N}\subset\Cb$ be a sequence with $f_n\downarrow f$, where $\Ub$
 consists of all upper semicontinuous functions $g\colon\R^d\to\Rb$ with 
 $\|g^+\|_\infty<\infty$. Due to Dini's theorem and 
 $\lim_{|x|\to\infty}H_+^*(x)=\infty$, the functionals
 \begin{align*}
  \Phi &\colon\Ub\to\Rb,\; g\mapsto\sup_{x\in\R^d}\big(f(x)-H_+^*(x)\big), \\
  \Phi_n &\colon\Ub\to\Rb,\; g\mapsto\frac{1}{n}\log\big(\e\big[e^{nf(X_n)}\big]\big)
  \quad\mbox{for all } n\in\N
 \end{align*}
 are continuous from above. Hence, by changing a supremum with an infimum at the
 cost of an inequality, we obtain
 \begin{align*}
  \limsup_{n\to\infty}\frac{1}{n}\log(\P(X_n\in A))
  &=\limsup_{n\to\infty}\Phi_n(f) 
  =\limsup_{n\to\infty}\inf_{k\in\N}\Phi_n(f_k) \\
  &\leq\inf_{k\in\N}\inf_{n\in\N}\sup_{l\geq n}\Phi_l(f_k) 
  =\inf_{k\in\N}\limsup_{n\to\infty}\Phi_n(f_k) \\
  &\leq\inf_{k\in\N}\Phi(f_k)=\Phi(f)=-\inf_{x\in A_L}\Lambda^*(x).  \qedhere
 \end{align*}
\end{proof}

Replacing the entropic risk measure by the short fall risk measure leads to polynomial
rather than exponential convergence rates.

\begin{theorem} \label{thm:cramer2}
 Assume that $\e[|\xi_1|^p]<\infty$ for some $p\in (1,\infty)$. Define
 \[ \Lambda\colon\R^d\to\R,\; x\mapsto\inf\{m\in\R\colon\e[((1+x\xi_1-m)^+)^p]\leq 1\} \]
 and the convex conjugate
 \[ \Lambda^*\colon\R^d\to\R,\; x\mapsto\sup_{y\in\R^d}(xy-\Lambda(y)). \]
 Then, for every closed subset $A\subset\R^d$, 
 \[ \limsup_{n\to\infty}n^{p-1}\P(X_n\in A)\leq\left(\inf_{x\in A_L}\Lambda^*(x)\right)^{-p}, \]
 where $A_L:=\{x+y\colon x\in A, \, y\in B(L)\}$. 
\end{theorem}
\begin{proof}
 We consider the short fall risk measure
 \[ \E\colon {\rm C}_{\tilde{\kappa}}\to\R,\; f\mapsto\inf\{m\in\R\colon\e[((1+f(\xi_1)-m)^+)^p]\leq 1\} \]
 which is indeed a convex expectation and continuous from above, see~\cite[Chapter~4.9]{FS16}. 
 Hence, Corollary~\ref{cor:LLN} yields a corresponding semigroup $(S(t))_{t\geq 0}$ with generator 
 \[ (Af)(x)=\Lambda(\nabla f(x))+\phi_0(x)\nabla f(x)
 	\quad\mbox{for all } f\in\Cb^1 \mbox{ and } x\in\R^d. \]
 Similar to the proof of Theorem~\ref{thm:cramer}, by using that the functionals
 $\Phi(f):=(S(1)f)(0)$ and $\Phi_n(f):=(I(\frac{1}{n})^n f)(0)$ are continuous from above 
 on $\Ub$, one can show that
 \[ \limsup_{n\to\infty}\big(I\big(\tfrac{1}{n}\big)^n f\big)(0)
 	\leq -\inf_{x\in A_L}\Lambda^*(x), \]
 where $f:=-\infty\one_{A^c}$ and $(I(t)f)(x):=t\E[\frac{1}{t}f(x+t\xi_1)]$. It remains to show that
 \[ -n^{-\frac{p-1}{p}}\P(X_n\in A)^{-\frac{1}{p}}\leq\big(I\big(\tfrac{1}{n}\big)^n f\big)(0)
 	\quad\mbox{for all } n\in\N. \]
 To do so, we want to apply~\cite[Lemma~4.2]{Lacker}. Let $\p$ be the set of all probability 
 measures on the Borel-$\sigma$-algebra $\B(\R^d)$. By Fenchel--Moreau's theorem, we have
 \[ \E[f]=\sup_{\nu\in\p}\left(\int_{\R^d} f\,\d\nu-\alpha(\nu)\right)
 	\quad\mbox{for all } f\in\Cb, \]
 where $\alpha(\nu):=\sup_{f\in\Cb}\big(\int_{\R^d} f\,\d\nu-\E[f]\big)$. By induction, one can
 show that
 \[ \big(I\big(\tfrac{1}{n}\big)^n f\big)(0)
 	=\sup_{\nu\in\p^n}\left(\int_{(\R^d)^n}f\big(\psi_n\big(\tfrac{1}{n},0,x_1,\ldots,x_n\big)\big)
 		\nu(\d x_1,\ldots,\d x_n)-\frac{1}{n}\alpha_n(\nu)\right) \]
 for all $n\in\N$ and $f\in\Cb$, where $\p^n$ consists of all probability measures on 
 the Borel-$\sigma$-algebra $\B((\R^d)^n)$. Moreover, the penalization functions
 $\alpha_n\colon\p^n\to [0,\infty]$ are defined by
 \[ \alpha_n(\nu):=\int_{(\R^d)^n}\sum_{i=1}^n \alpha(\nu_{i-1,i}(x_1,\ldots,x_{i-1}))\,\nu(\d x_1,\ldots, \d x_n), \]
 where the kernels $\nu_{i-1,i}$ are determined by the disintegration 
 \[ \nu(\d x_1,\ldots,\d x_n)=\nu_{0,1}(\d x_1)\prod_{i=2}^n \nu_{i-1,i}(x_1,\ldots,x_{i-1})(\d x_i). \]
 It follows from~\cite[Lemma~4.2]{Lacker} that
 \[ \alpha_n(\nu)\leq n^\frac{1}{p}\left\|\frac{\d\nu}{\d\mu^n}\right\|_{L^q(\mu^n)}
 	\quad\mbox{for all } \nu\ll\mu, \]
 where $\mu:=\P\circ\xi_1^{-1}$ and $\mu^n:=\mu\otimes\ldots\otimes\mu$ denotes 
 the $n$-fold product measure. Now, the claim follows similarly to the proof 
 of~\cite[Theorem~1.2]{Lacker}. 
\end{proof}

\subsection{Uncertain samples in Wasserstein space}
\label{sec:Wasser}

As an illustration of the previous abstract results, we consider convex expectations that 
are defined as a supremum over a set of probability measures which are weighted according 
to their Wasserstein distance to a fixed reference model. This type of uncertainty has previously 
been studied in a framework with nonlinear semigroups corresponding to Markov processes 
with uncertain transition probabilities, see~\cite{BDKN,FKN}. Let $p\in (1,\infty)$ and denote 
by $\p_p$ the $p$-Wasserstein space consisting of all probability measures on the 
Borel-$\sigma$-algebra $\B(\R^d)$ with finite $p$-th moment. We endow $\p_p$ with the 
$p$-Wasserstein distance 
\[ \W_p(\mu,\nu):=\left(\inf_{\pi\in\Pi(\mu,\nu)}\int_{\R^d\times\R^d}
	|x-y|^{\frac{1}{p}}\,\pi(\d x,\d y)\right)^{\frac{1}{p}}
	\quad\mbox{for all } \mu,\nu\in\p_p, \]
where $\Pi(\mu,\nu)$ consists of all probability measures on $\B(\R^d\times\R^d)$
with first marginal $\mu$ and second marginal $\nu$. Furthermore, let 
$\phi\colon\R_+\to [0,\infty]$ be a function with $\phi(0)=0$ and 
$\lim_{c\to\infty}\nicefrac{\phi(c)}{c}=\infty$. In the sequel, we fix $\mu\in\p_p$ 
and define
\begin{equation} \label{eq:Wasser.E}
 \E\colon {\rm C}_{\tilde{\kappa}}\to\R,\; f\mapsto\sup_{\nu\in\p_p}
 \left(\int_{\R^d}f(x)\,\nu(\d x)-\phi(\W_p(\mu,\nu))\right).
\end{equation}
In the case $\phi:=\infty\one_{[0,r]}$, the functional $\E$ is defined as supremum 
over all measures in a Wasserstein ball with radius $r\geq 0$ around the reference 
model $\mu$. Moreover, for every $t>0$, $f\in\Cb$ and $x\in\R^d$, 
\begin{align*}
 t\E\left[\frac{1}{t}f(x+t\xi)\right] 
 &=\sup_{\nu\in\p_p}\left(\int_{\R^d} f(x+ty)\,\nu(\d y)-t\phi(\W_p(\mu,\nu))\right) \\
 &=\sup_{\nu\in\p_p}
 \left(\int_{\R^d}f(x+y)\,\nu_t(\d y)-t\phi\left(\frac{\W_p(\mu_t,\nu_t)}{t}\right)\right) \\
 &=\sup_{\nu\in\p_p}
 \left(\int_{\R^d} f(x+y)\,\nu(\d y)-t\phi\left(\frac{\W_p(\mu_t,\nu)}{t}\right)\right),
\end{align*}
where $\xi:=\id_{\R^d}$, $\nu_t:=\nu\circ (t\xi)^{-1}$ and $\mu_t:=\mu\circ (t\xi)^{-1}$. 
Hence, due to the expression in the last line, the definition
\[ (I(t)f)(x):=t\E\left[\frac{1}{t}f(x+t\xi)\right] \]
is consistent with the scaling of the penalization function $\phi$ in~\cite{BDKN,FKN}. 
In the following theorem, choosing $p,q\in (1,\infty)$ with $1/p+1/q=1$,
$\varphi(c):=c^p/p$ and $m=0$ yields the generator $(Af)(x)=|\nabla f(x)|^q/q$
for all $f\in\Cb^1$ and $x\in\R^d$.

\begin{theorem} \label{thm:Wasser}
 The functional $\E$ defined by equation~\eqref{eq:Wasser.E} is a convex expectation which 
 is continuous from above. Hence, there exists a strongly continuous convex monotone semigroup $(S(t))_{t\geq 0}$ 
 on $\Cb$ with
 \[ S(t)f=\lim_{n\to\infty}I\big(\tfrac{t}{n}\big)^n f\in\Cb
 	\quad\mbox{for all } t\geq 0 \mbox{ and } f\in\Cb \]
 which has the properties~(i)-(v) from Theorem~\ref{thm:cher} with $\omega=L=0$.
 For every $f\in\Cb^1$, it holds that $f\in D(A)$ and 
 \[ (Af)(x)=\sup_{c\geq 0}(c|\nabla f(x)|-\phi(c))+m\nabla f(x)
 	\quad\mbox{for all } x\in\R^d, \]
 where $m:=\int_{\R^d}y\,\mu(\d y)$. Furthermore, let $(\Omega,\H,\bar{\E})$ be a 
 convex expectation space and $(\xi_n)_{n\in\N}$ be an iid sequence of random 
 vectors $\xi_n\colon\Omega\to\R^d$ with $\bar{\E}[f(\xi_n)]=\E[f]$ for all $f\in\Cb$. 
 Then, for every $f\in\Cb$, $t>0$ and $x\in\R^d$, 
 \[ \lim_{n\to\infty}\frac{t}{n}\bar{\E}\left[\frac{n}{t}f\Big(x+\frac{t}{n}\sum_{i=1}^n \xi_i\Big)\right]
 	=\sup_{y\in\R^d}\big(f(x+t(m+y))-\phi(|y|)t\big). \]
\end{theorem}
\begin{proof}
 First, we show that $\E$ is a convex expectation which is continuous from above. 
 Let $f\in {\rm C}_{\tilde{\kappa}}$ and choose $c\geq 0$ with $|f(x)|\leq c(1+|x|)$ for all $x\in\R^d$. 
 We use
 \[ \left|\int_{\R^d}|x|\,\mu(\d x)-\int_{\R^d}|y|\, \nu(\d y)\right|
 	\leq\int_{\R^d\times\R^d}\big||x|-|y|\big|\,\pi(\d x, \d y)
 	\leq\int_{\R^d\times\R^d}|x-y|\,\pi(\d x, \d y) \]
 for all $\pi\in\Pi(\mu,\nu)$, H\"older's inequality and 
 $\lim_{c\to\infty}\nicefrac{\phi(c)}{c}=\infty$ to estimate
 \begin{align*}
  \E[|f|]
  &\leq\sup_{\nu\in\p_p}\left(\int_{\R^d}c(1+|x|)\, \nu(\d x)-\phi(\W_p(\mu,\nu))\right) \\
  &=\sup_{\nu\in\p_p}\left(\int_{\R^d}c(1+|x|)\, \mu(\d x)+\int_{\R^d}c|x|\, \nu(\d x)
  -\int_{\R^d}c|x|\, \mu(\d x)-\phi(\W_p(\mu,\nu))\right) \\
  &\leq\int_{\R^d}c(1+|x|)\,\mu(\d x)
  +\sup_{\nu\in\p_p}\left(c\W_1(\mu,\nu)-\phi(\W_p(\mu,\nu))\right) \\
  &\leq\int_{\R^d}c(1+|x|)\,\mu(\d x)
  +\sup_{\nu\in\p_p}\left(c\W_p(\mu,\nu)-\phi(\W_p(\mu,\nu))\right)<\infty.
 \end{align*}
 It follows from $\phi(0)=0$ that $\E[c]=c$ for all $c\in\R$. Moreover, $\E$ is clearly
 convex and monotone. Let $(f_n)_{n\in\N}\subset {\rm C}_{\tilde{\kappa}}$ with $f_n\downarrow 0$
 and choose $c\geq 0$ with $f_1\leq c(1+|x|)$. Then, 
 \[ \int_{\R^d} f_n(x)\,\nu(\d x)-\phi(\W_p(\mu,\nu))
 	\leq\int_{\R^d}c(1+|x|)\,\mu(\d x)+c\W_p(\mu,\nu)-\phi(\W_p(\mu,\nu)) \]
 for all $\nu\in\p_p$. Since $\lim_{c\to\infty}\nicefrac{\phi(c)}{c}=\infty$, there exists 
 $R\geq 0$ with
 \[ \E[f_n]=\sup_{\nu\in M}\left(\int_{\R^d} f_n(x)\,\nu(\d x)-\phi(\W_p(\mu,\nu))\right)
 	\leq\sup_{\nu\in M}\int_{\R^d} f_n(x)\,\nu(\d x), \]
 where $M:=\{\nu\in\p_p\colon\W_p(\mu,\nu)\leq R\}$. For every $\epsilon>0$,
 Lemma~\ref{lem:tightW} implies that there exists $r\geq 0$ with 
 \[ \sup_{\nu\in M}\int_{B(r)^c}c(1+|x|)\,\nu(\d x)\leq\frac{\epsilon}{2}. \]
 Moreover, we can use Dini's theorem to choose $n_0\in\N$ with
 \[ \int_{\R^d} f_n(x)\,\nu(\d x)
 	\leq\int_{B(r)}f_n(x)\,\nu(\d x)+\int_{B(r)^c}c(1+|x|)\,\nu(\d x)\leq\epsilon \]
 for all $n\geq n_0$ and $\nu\in M$. We obtain $\E[f_n]\downarrow 0$ as $n\to\infty$. 
 Now, Theorem~\ref{thm:LLN} yields the existence of the semigroup $(S(t))_{t\geq 0}$.

 Second, for every $x\in\R^d$, we show that
 \[ \E[x\xi]=\sup_{c\geq 0}(c|x|-\phi(c))+\int_{\R^d}xy\,\mu(\d y). \]
 W.l.o.g., let $x\neq 0$. For every $c\geq 0$, we choose $\nu:=\mu*\delta_\frac{cx}{|x|}$. 
 Then, 
 \[ \int_{\R^d}xy\,\nu(\d y)
 	=\int_{\R^d}\int_{\R^d}x(y+z)\,\mu(\d y)\delta_\frac{cx}{|x|}(\d z)
 	=c|x|+\int_{\R^d}xy\,\mu(\d y). \]
 We take the supremum over $c\geq 0$ and use $\W_p(\mu,\nu)=c$ to conclude
 \[ \E[x\xi]\geq\sup_{c\geq 0}(c|x|-\phi(c))+\int_{\R^d}xy\,\mu(\d y). \]
 For every $\nu\in\p_p$, it follows from
 \[ \left|\int_{\R^d}xy\,\mu(\d y)-\int_{\R^d}xz\,\nu(\d z)\right|
 	\leq |x|\int_{\R^d\times\R^d}|y-z|\,\pi(\d y, \d z) \]
 for all $\pi\in\Pi(\mu,\nu)$ and H\"older's inequality that 
 \begin{align*}
  \int_{\R^d}xy\,\nu(\d y)-\phi(\W_p(\mu,\nu))
  &=\int_{\R^d}xy\,\mu(\d y)+\int_{\R^d}xy\,\nu(\d y)-\int_{\R^d}xy\,\mu(\d y)-\phi(c) \\
  &\leq\int_{\R^d}xy\,\mu(\d y)+\W_1(\mu,\nu)|x|-\phi(c) \\
  &\leq\int_{\R^d}xy\,\mu(\d y)+c|x|-\phi(c), 
 \end{align*}
 where $c:=\W_p(\mu,\nu)$. 
 Taking the supremum over $\nu\in\p_p$ yields
 \[ \E[x\xi]\leq\sup_{c\geq 0}(c|x|-\phi(c))+\int_{\R^d}xy\,\mu(\d y). \]
 In particular, for every $f\in\Cb^1$ and $x\in\R^d$, we obtain from Theorem~\ref{thm:LLN2}
 that
 \[ (Af)(x)=\E[\nabla f(x)\xi]
 	=\sup_{c\geq 0}(c|\nabla f(x)|-\phi(c))+m\nabla f(x), \]
 where $m:=\int_{\R^d}y\,\mu(\d y)$. Moreover, for every $x\in\R^d$, 
 \[ \E[x\xi]=\sup_{y\in\R^d}(xy-\psi(y)),  \quad\mbox{where } \psi(y):=\phi(|y-m|). \]
 Hence, Fenchel--Moreau's theorem and Theorem~\ref{thm:LLN2} imply
 \[ (S(t)f)(x)=\sup_{y\in\R^d}\big(f(x+t(m+y))-\phi(|y|)\big) \]
 for all $t\geq 0$, $f\in\Cb$ and $x\in\R^d$. We apply again Theorem~\ref{thm:LLN2}
 to obtain the last part of the statement. 
\end{proof}

We conclude this subsection by showing that the semigroup (and thus the distribution) 
which we obtain in the limit is the same whether we define the convex expectation as 
the supremum over an uncertainty set of measures or a set of parameters in $\R^d$. 
The latter corresponds to shifting the fixed measures $\mu$ in all possible deterministic 
directions. We define
\[ \tilde{\E}\colon {\rm C}_{\tilde{\kappa}}\to\R,\; f\mapsto\sup_{\lambda\in\R^d}
	\left(\int_{\R^d} f(x+\lambda)\,\mu(\d x)-\phi(|\lambda|)\right), \]
$J(0):=\id_{\Cb}$ and, for every $t>0$, $f\in\Cb$ and $x\in\R^d$,
\[ (J(t)f)(x):=t\tilde{\E}\left[\frac{1}{t}f(x+t\xi)\right]. \]

\begin{corollary}
 Denoting by $(S(t))_{t\geq 0}$ the semigroup from Theorem~\ref{thm:Wasser},
 we have 
 \[ S(t)f=\lim_{n\to\infty}J\big(\tfrac{t}{n}\big)^n f
 	\quad\mbox{for all } t\geq 0 \mbox{ and } f\in\Cb. \]
\end{corollary} 
\begin{proof}
 It follows from $\W_p(\mu,\nu)=|\lambda|$ for $\nu:=\mu*\delta_\lambda$ with 
 $\lambda\in\R^d$ that $\tilde{\E}[|f|]\leq\E[|f|]$ for all $f\in {\rm C}_{\tilde{\kappa}}$. 
 Hence, $\tilde{\E}\colon {\rm C}_{\tilde{\kappa}}\to\R$ is a well-defined convex expectation 
 which is continuous from above. By Theorem~\ref{thm:LLN2}, there exists a semigroup $(T(t))_{t\geq 0}$
 on ${\rm C}_{\tilde{\kappa}}$ with 
 \[ T(t)f=\lim_{n\to\infty}J\big(\tfrac{t}{n}\big)^n f 
 	\quad\mbox{for all } t\geq 0 \mbox{ and } f\in\Cb \]
 and generator $(Bf)(x)=\tilde{\E}[\nabla f(x)\xi]$ for all $f\in\Cb^1$ and $x\in\R^d$. 
 By Theorem~\ref{thm:Wasser} and a straightforward computation, it holds that
 \[ (Af)(x)=\sup_{c\geq 0}\big(c|\nabla f(x)|-\phi(c)\big)+m\nabla f(x)=(Bf)(x) \]
 for all $f\in\Cb^1$ and $x\in\R^d$. Hence, Corollary~\ref{cor:IJ} implies $S(t)f=T(t)f$ 
 for all $t\geq 0$ and $f\in\Cb$.
\end{proof}

\section{Second order scaling limits}
\label{sec:CLT}

Throughout this section, we choose the weight function
\[ \kappa\colon\R^d\to (0,\infty),\; x\mapsto (1+|x|)^{-2}. \]
In analogy to the previous section, let $(\xi_n)_{n\in\N}\subset\H^d$ be a sequence
of iid random vectors on a convex expectation space $(\Omega,\H,\bar{\E})$ with 
finite second moments and $\bar{\E}[a\xi_1]=0$ for all $a\in\R^d$. For the sake of readability, 
we refrain from introducing a sequence $(\psi_n)_{n\in\N}$ of recursively defined functions 
as in Section~\ref{sec:LLN} and study only the limit behaviour of the sequence 
\[ X_n:=\frac{1}{\sqrt{n}}\sum_{i=1}^n \xi_i \]
which corresponds to choosing $\psi(t,x,y):=x+\sqrt{t}y$. Again, we are interested in the 
behaviour of the distributions of $X_n$ and requiring finite second moments means that 
the distribution of $\xi_1$ is well-defined even for continuous functions with at most 
quadratic growth at infinity, i.e., 
\[ F_{\xi_1}\colon {\rm C}_{\tilde{\kappa}}\to\R,\; f\mapsto\bar{\E}[f(\xi_1)]. \]
Denoting this functional by $\E[\,\cdot\,]$, we remark that continuity from above on 
${\rm C}_{\tilde{\kappa}}$ is equivalent to the uniform integrability condition 
$\lim_{c\to\infty}\bar{\E}[(|\xi_1|^2-c)^+]=0$ from~\cite{Peng19}. Defining
\[ (I(t)f)(x):=t\bar{\E}\left[\frac{1}{t}f(x+\sqrt{t}\xi_1)\right] \]
for all $t>0$, $f\in\Cb$ and $x\in\R^d$ leads to
\[ (S(1)f)(0)=\lim_{n\to\infty}\big(I\big(\tfrac{1}{n}\big)^n f\big)(0)
	=\lim_{n\to\infty}\frac{1}{n}\bar{\E}\left[nf\left(\frac{1}{\sqrt{n}}\sum_{i=1}^n \xi_i\right)\right]. \]
This result is proven in Subsection~\ref{sec:CLT1} analogously to Theorem~\ref{thm:LLN}.
Hence, it seems very likely that a more general result including a sequence $(\psi_n)_{n\in\N}$ 
of recursively defined functions can be obtained mainly at the cost of additional terms appearing
in the estimates which would make the proof quite lengthy and might disguise the main
differences between the first and second order case. In Subsection~\ref{sec:Wasser2},
we illustrate again the abstract results by considering convex expectations that are defined 
as a (weighted) supremum over a set of probability measures. The necessary modifications of 
the setting will be explained in full detail. We also want to mention that this approach allows for a 
generalization of the previous results from~\cite{BEK, FKN} about Markov processes with uncertain 
transition probabilities, which were restricted to first order perturbations, to the second order case. 
The possibility of this extension was already conjectured in~\cite{BEK} but whether the original 
construction of the semigroup based on monotone convergence can be transferred remains an 
open question.

\subsection{CLT for convex expectations}
\label{sec:CLT1}

Let $\E\colon {\rm C}_{\tilde{\kappa}}\to\R$ be a convex expectation which is continuous from above. 
By~\cite[Theorem~C.1]{BDKN}, there exists a convex monotone extension 
$\E\colon {\rm B}_{\tilde{\kappa}}\to\R$ such that, for every $\epsilon>0$ and $c\geq 0$, 
there exists $K\Subset\R^d$ with 
\begin{equation} \label{eq:tight2}
 \E\left[\tfrac{c}{\kappa}\one_{K^c}\right]<\epsilon.
\end{equation}
Here, ${\rm B}_{\tilde{\kappa}}$ denotes the space of all Borel measurable functions $f\colon\R^d\to\R$ 
with $\|f\|_{\tilde{\kappa}}<\infty$. We define $I(0):=\id_{\Cb}$ and, for every $t>0$, $f\in\Cb$ 
and $x\in\R^d$,  
\[ (I(t)f)(x):=t\E\left[\frac{1}{t}f(x+\sqrt{t}\xi)\right], 
	\quad\mbox{where}\quad \xi:=\id_{\R^d}. \]
Moreover, we denote by $\Cb^2$ the space of all twice differentiable functions $f\in\Cb$ 
such that all partial derivatives are in $\Cb$ and by 
$D^2 f(x)=(\partial_i\partial_j f)_{i,j=1,\ldots,d}\in\R^{d\times d}$ the second derivative. 
Let $x^T\in\R^{1\times d}$ be the transposed vector for all $x\in\R^d\cong\R^{d\times 1}$. 
The space $\R^{d\times d}$ is subsequently endowed with the Frobenius norm.

\begin{theorem} \label{thm:CLT}
 Assume that $\E[a\xi]=0$ for all $a\in\R^d$. Then, there exists a strongly continuous 
 convex monotone semigroup $(S(t))_{t\geq 0}$ on $\Cb$ with
 \begin{equation} \label{eq:CLT}
  S(t)f=\lim_{n\to\infty}I\big(\tfrac{t}{n}\big)^n f\in\Cb
  \quad\mbox{for all } t\geq 0 \mbox{ and } f\in\Cb
 \end{equation}
 which has the properties~(i)-(v) from Theorem~\ref{thm:cher} with $\omega=L=0$. 
 For every $f\in\Cb^2$ and $x\in\R^d$, it holds that $f\in D(A)$ and 
 \[ (Af)(x)=\E\left[\frac{1}{2}\xi^T D^2f(x)\xi\right]. \]
 In addition, for every convex expectation space $(\Omega,\H,\bar{\E})$ and iid sequence
 $(\xi_n)_{n\in\N}\subset\H^d$ with $\bar{\E}[f(\xi_n)]=\E[f]$ for all $f\in\Cb$, 
 \begin{equation} \label{eq:CLT2}
  (S(1)f)(0)=\lim_{n\to\infty}\frac{1}{n}\bar{\E}\left[nf\left(\frac{1}{\sqrt{n}}\sum_{i=1}^n \xi_i\right)\right]
  \quad\mbox{for all } f\in\Cb. 
 \end{equation}
\end{theorem}
\begin{proof}
 It follows from Lemma~\ref{lem:E}(ii) that Assumption~\ref{ass:cher}(i)-(iv) and~(vii) 
 are satisfied with $\omega=L=0$. Moreover, similar to the proof of Theorem~\ref{thm:LLN}, 
 one can show that $\|I(t)f)\|_{\tilde{\kappa}}\leq e^{ct}\|f\|_{\tilde{\kappa}}$ for all $f\in\Cb$ with
 $\|f\|_{\tilde{\kappa}}\leq 1$ and $t\in (0,1]$, where $c:=\frac{1}{2}\E[2|\xi|^2]$. Hence, 
 Remark~\ref{rem:cher}(b) yields that Assumption~\ref{ass:cher}(vi) is satisfied. 
 Next, we show that 
 \[ (I'(0)f)(x)=\E\left[\frac{1}{2}\xi^T D^2 f(x)\xi\right]
 	\quad\mbox{for all } f\in\Cb^2 \mbox{ and } x\in\R^d. \]
 For every $h>0$ and $x\in\R^d$, 
 \[ f(x+\sqrt{h}\xi)=f(x)+\nabla f(x)\sqrt{h}\xi
 	+\int_0^1\int_0^1 sh\xi^T D^2 f(x+rs\sqrt{h}\xi)\xi\,\d r\,\d s. \]
 Hence, it follows from Lemma~\ref{lem:E}(vi) and Lemma~\ref{lem:lambda} that
 \begin{align*}
  &\left(\frac{I(h)f-f}{h}\right)(x)-\E\left[\frac{1}{2}\xi^T D^2 f(x)\xi\right] \\
  &=\E\left[\frac{1}{\sqrt{h}}\nabla f(x)\xi+\int_0^1\int_0^1 s\xi^T D^2 f(x+rs\sqrt{h}\xi)\xi\,\d r\,\d s\right]
  -\E\left[\frac{1}{2}\xi^T D^2 f(x)\xi\right] \\
  &\leq\lambda\E\left[\frac{1}{\lambda}
  \left(\int_0^1\int_0^1 s\xi^T D^2 f(x+rs\sqrt{h}\xi)\xi\,\d r\,\d s-\frac{1}{2}\xi^T D^2 f(x)\xi\right)
  +\frac{1}{2}\xi^T D^2 f(x)\xi\right] \\
  &\quad\, -\lambda\E\left[\frac{1}{2}\xi^T D^2 f(x)\xi\right] \\
  &\leq\frac{\lambda}{2}\E\left[\frac{2}{\lambda}
  \left(\int_0^1\int_0^1 s\xi^T D^2 f(x+rs\sqrt{h}\xi)\xi\,\d r\,\d s-\frac{1}{2}\xi^T D^2 f(x)\xi\right)\right] \\
  &\quad\; +\frac{\lambda}{2}\E\left[\xi^T D^2 f(x)\xi\right]-\lambda\E\left[\frac{1}{2}\xi^T D^2 f(x)\xi\right].
 \end{align*}
 Let $\epsilon>0$ and choose $\lambda\in (0,1]$ with $\E[\|D^2 f\|_\infty |\xi|^2]<\nicefrac{\epsilon}{2}$. 
 Lemma~\ref{lem:E}(iii) and~(v) imply
 \[ \frac{\lambda}{2}\E\left[\xi^T D^2 f(x)\xi\right]-\lambda\E\left[\frac{1}{2}\xi^T D^2 f(x)\xi\right]
 	\leq\frac{\epsilon}{2}. \]
 Furthermore, by inequality~\eqref{eq:tight2}, there exists $n\in\N$ with
 \begin{align*}
  &\frac{\lambda}{2}\E\left[\frac{2}{\lambda}
  \left(\int_0^1\int_0^1 s\xi^T D^2 f(x+rs\sqrt{h}\xi)\xi\,\d r\,\d s-\frac{1}{2}\xi^T D^2 f(x)\xi\right)
  \one_{B(n)^c}(\xi)\right] \\
  &\leq\frac{\lambda}{2}\E\left[\frac{2}{\lambda}
  \left(\int_0^1 s\, \d s+\frac{1}{2}\right)\|D^2 f\|_\infty |\xi|^2\one_{B(n)^c}(\xi)\right]
  \leq\frac{\epsilon}{4}
 \end{align*}
 for all $x\in\R^d$. Now, let $K\Subset\R^d$. Since $D^2 f$ is uniformly continuous 
 on compacts, there exists $\delta>0$ with
 \[ |D^2 f(x+y)-D^2 f(x)|<\frac{\epsilon}{2n^2}
 	\quad\mbox{for all } x\in K \mbox{ and } y\in B(\delta). \]
 Then, for every $h\in (0,\nicefrac{\delta^2}{n^2}]$ and $x\in K$, 
 \begin{align*}
  &\frac{\lambda}{2}\E\left[\frac{2}{\lambda}
  \left(\int_0^1\int_0^1 s\xi^T D^2 f(x+rs\sqrt{h}\xi)\xi\,\d r\,\d s-\frac{1}{2}\xi^T D^2 f(x)\xi\right)
  \one_{B(n)}(\xi)\right] \\
  &=\frac{\lambda}{2}\E\left[\frac{2}{\lambda}
  \left(\int_0^1\int_0^1 s\xi^T\big(D^2 f(x+rs\sqrt{h}\xi)-D^2 f(x)\big)\xi\,\d r\,\d s\right)
  \one_{B(n)}(\xi)\right] \\
  &^\leq\frac{\lambda}{2}\E\left[\frac{2}{\lambda}
   \left(\int_0^1 s|\xi|^2\frac{\epsilon}{2n^2}\,\d s\right)\one_{B(n)}(\xi)\right]
   \leq\frac{\epsilon}{4}. 
 \end{align*} 
 Hence, for every $\epsilon>0$ and $K\Subset\R^d$, there exists $h_0>0$ with
 \[ \left(\frac{I(h)f-f}{h}\right)(x)-\E\left[\frac{1}{2}\xi^T D^2 f(x)\xi\right]\leq\epsilon \]
 for all $x\in K$ and $h\in (0,h_0]$. The lower bound follows by similar arguments. 
 Furthermore,
 \[ \|I(t)f-f\|_\infty\leq\E\left[\frac{1}{2}\|D^2 f\|_\infty |\xi|^2\right]t 
 	\quad\mbox{for all } t\geq 0. \]
 This shows that Assumption~\ref{ass:cher}(v) is satisfied. Now, the first part of the
 claim follows from Theorem~\ref{thm:cher} and the second part follows similarly
 to the proof of Theorem~\ref{thm:LLN}.
\end{proof}

Similar to Theorem~\ref{thm:LLN}, for a sublinear expectation $\E[\,\cdot\,]$,
the computation of the generator simplifies but this does not affect the result.
Furthermore, in the sublinear case, the semigroup $(S(t))_{t\geq 0}$ is the unique
viscosity solution of the PDE $u_t=G(D^2 u)$ with $G(a):=\E[\frac{1}{2}\xi^T a\xi]$
and therefore Theorem~\ref{thm:CLT} is consistent with previous results in that direction.
If we weaken the condition $\E[a\xi]=0$ for all $a\in\R^d$ by merely requiring
\begin{equation} \label{eq:mean}
 \E[a\xi]\geq 0 \quad\mbox{for all } a\in\R^d,
\end{equation}
one can still apply the previous result on the transformed expectation
\begin{equation} \label{eq:mean2}
 \tilde{\E}\colon\H\to\R,\; X\mapsto\inf_{a\in\R^d}\E[X+a\xi]. 
\end{equation}
In the particular case $\E[f]=\sup_{\mu\in M}\int_{\R^d} f(x)\,\mu(\d x)$ for a set $M$ 
of probability measures on $(\R^d,\B(\R^d))$, condition~\eqref{eq:mean} is satisfied if 
$\int_{\R^d} x\,\mu(\d x)=0$ for some $\mu\in M$.

\begin{corollary}
 Assume that $\E[a\xi]\geq 0$ for all $a\in\R^d$. Define $I(0):=\id_{\Cb}$ and
 \[ (I(t)f)(x):=\inf_{a\in\R^d}t\E\left[\frac{1}{t}f(x+\sqrt{t}\xi)+a\xi\right] \]
 for all $t>0$, $f\in\Cb$ and $x\in\R^d$. Then, there exists a strongly continuous 
 convex monotone semigroup $(S(t))_{t\geq 0}$ on $\Cb$ with
 \[ S(t)f=\lim_{n\to\infty}I\big(\tfrac{t}{n}\big)^n f\in\Cb
 	\quad\mbox{for all } t\geq 0 \mbox{ and } f\in\Cb \]
 which has the properties~(i)-(v) from Theorem~\ref{thm:cher} with $\omega=L=0$. 
 For every $f\in\Cb^2$ and $x\in\R^d$, it holds that $f\in D(A)$ and 
 \[ (Af)(x)=\inf_{a\in\R^d}\E\left[\frac{1}{2}\big(\xi^T D^2f(x)\xi+a\xi\big)\right]. \]
\end{corollary}
\begin{proof}
 By Lemma~\ref{lem:mean}, the convex expectation $\tilde{\E}$ defined by 
 equation~\eqref{eq:mean2} satisfies $\tilde{\E}[a\xi]=0$ for all $a\in\R^d$. Hence, 
 we can apply Theorem~\ref{thm:CLT} to obtain the claim. 
\end{proof}

\subsection{Uncertain samples in Wasserstein space}
\label{sec:Wasser2}

Similar to Subsection~\ref{sec:Wasser}, we consider convex expectations that 
are defined as a supremum over a set of probability measures but some natural 
modifications of the setting are necessary. To ensure that the convex expectation
is continuous from above on functions with quadratic growth at infinity, let $p>2$ 
and $\phi\colon [0,\infty]\to [0,\infty]$ be a non-decreasing function with $\phi(0)=0$,
$\phi(\infty)=\infty$ and $\lim_{c\to\infty}\nicefrac{\phi(c)}{c^2}=\infty$. Moreover,
we fix a reference measure $\mu\in\p_p$ with mean zero, i.e., $\int_{\R^d}x\,\mu(\d x)=0$. 
In view of Lemma~\ref{lem:mean} it now seems natural to define
\[ \tilde{\E}\colon {\rm C}_{\tilde{\kappa}}\to\R,\; f\mapsto\inf_{a\in\R^d}\sup_{\nu\in\p_p}
	\left(\int_{\R^d}f(x)+ax\,\nu(\d x)-\phi(\W_p(\mu,\nu))\right). \]
Using~\cite[Theorem~2]{Fan} to interchange the supremum with the infimum, one 
can show that 
\[ \tilde{\E}[f]=\sup_{\nu\in\p_p^0}\left(\int_{\R^d}f(x)\,\nu(\d x)-\phi(\W_p(\mu,\nu))\right)
	\quad\mbox{for all } f\in {\rm C}_{\tilde{\kappa}}, \]
where
\[ \p_p^0:=\left\{\nu\in\p_p\colon\int_{\R^d} x\,\nu(\d x)=0\right\}. \]
Furthermore, Theorem~\ref{thm:CLT} yields a corresponding semigroup $(S(t))_{t\geq 0}$
with generator 
\[ (Af)(x)=\E\left[\frac{1}{2}\xi^T D^2 f(x)\xi\right]
	\quad\mbox{for all } f\in\Cb^2 \mbox{ and } x\in\R^d. \]
However, if we want to give an explicit formula for the generator, i.e., the generator should 
be a given as a supremum over a set of parameters in $\R^d$ rather than a set of measures, 
it seems necessary to replace the Wasserstein distance $\W_p(\mu,\nu)$ by a transport
cost which is given as the infimum over a smaller set of couplings. One possible natural 
choice are martingale couplings~\cite{BHLP,BJ}. We call $\pi\in\Pi(\mu,\nu)$ a martingale 
coupling between $\mu$ and $\nu$ if there exist random variables $X,Y$ on a probability 
space $(\Omega,\F,\P)$ with $\pi=\P\circ (X,Y)^{-1}$ and $\e[Y|X]=X$. Equivalently, we 
could require that
\begin{equation} \label{eq:mart}
 \int_{\R^d} f(x)(y-x)\,\pi(\d x,\d y)=0 \quad\mbox{for all } f\in\Cb.
\end{equation}
Denoting by $\Pi_M(\mu,\nu)$ the set of of all martingale couplings between $\mu$ and $\nu$,
it follows from Strassen's theorem that $\Pi_M(\mu,\nu)\neq\emptyset$ if and only if $\mu\leq\nu$
in convex order. Moreover, we define the corresponding transport cost by
\[ C_M(\mu,\nu)
 	:=\left(\inf_{\pi\in\Pi_M(\mu,\nu)}\int_{\R^d\times\R^d}|x-y|^p\,\pi(\d x,\d y)\right)^\frac{1}{p}, \]
where $\inf\emptyset=\infty$. The convex expectation is then defined by
\begin{equation} \label{eq:Wasser2.E}
 \E\colon {\rm C}_{\tilde{\kappa}}\to\R,\; f\mapsto\sup_{\nu\in\p_p^0}
 \left(\int_{\R^d}f(x)\,\nu(\d x)-\phi(C_M(\mu,\nu))\right). 
\end{equation}
As before, we define $I(0):=\id_{\Cb}$ and, for every $t>0$, $f\in\Cb$ and $x\in\R^d$, 
\[ (I(t)f)(x):=t\E\left[\frac{1}{t}f(x+\sqrt{t}\xi)\right], 
	\quad\mbox{where } \xi:=\id_{\R^d}. \]
Similar to Subsection~\ref{sec:Wasser}, one can show that
\[ (I(t)f)(x)=\sup_{\nu\in\p_p^0}
	\left(\int_{\R^d} f(x+y)\,\nu(\d y)-t\phi\left(\frac{C_M(\mu_t,\nu)}{\sqrt{t}}\right)\right) \]
for all $t>0$, $f\in\Cb$ and $x\in\R^d$, where $\mu_t:=\mu\circ(\sqrt{t}\xi)^{-1}$. Let 
$\tr(a):=\sum_{i=1}^n a_{ii}$ be the matrix trace for all $a\in\R^{d\times d}$. Note that 
$\tr(\lambda\lambda^T a)=\lambda^T a\lambda$ and $\lambda\lambda^T\in\S_d^+$ 
for all $\lambda\in\R^d$ and $a\in\R^{d\times d}$, where $\S_d^+$ consists of all 
positive semi-definite symmetric $d\times d$-matrices. The proof of the following 
theorem is omitted since it is almost identical with the one of Theorem~\ref{thm:Wasser3} 
below. For details, we refer to discussion after Theorem~\ref{thm:Wasser3}.
In Theorem~\ref{thm:Wasser2} and Theorem~\ref{thm:Wasser3}, choosing $d=1$, $p,q\in (1,\infty)$ 
with $1/p+1/q=1$, $\varphi(c):=c^{2p}/p$ and $\Sigma:=1$ yields the generator 
$(Af)(x)=\frac{1}{2}f''(x)+(f''(x)^+)^q/q $ for all $f\in\Cb^2$ and $x\in\R$.

\begin{theorem} \label{thm:Wasser2}
 The functional $\E$ defined by equation~\eqref{eq:Wasser2.E} is a convex expectation which 
 is continuous from above and satisfies $\E[a\xi]=0$ for all $a\in\R^d$. Hence, there exists 
 a strongly continuous convex monotone semigroup $(S(t))_{t\geq 0}$ on $\Cb$ with
 \[ S(t)f=\lim_{n\to\infty}I\big(\tfrac{t}{n}\big)^n f\in\Cb
 	\quad\mbox{for all } t\geq 0 \mbox{ and } f\in\Cb \]
 which has the properties~(i)-(v) from Theorem~\ref{thm:cher} with $\omega=L=0$.
 For every $f\in\Cb^2$ and $x\in\R^d$, it holds that $f\in D(A)$ and 
 \[ (Af)(x)=\sup_{\lambda\in\R^d}
 	\left(\frac{1}{2}\tr(\lambda\lambda^T D^2f(x))-\phi(|\lambda|)\right)
 	+\frac{1}{2}\tr(\Sigma D^2f(x)), \]
 where $\Sigma:=\int_{\R^d} yy^T\,\mu(\d y)\in\S_d^+$. Moreover, for every convex 
 expectation space $(\Omega,\H,\bar{\E})$ and iid sequence $(\xi_n)_{n\in\N}\subset\H^d$ 
 with $\bar{\E}[f(\xi_n)]=\E[f]$ for all $f\in\Cb$, 
 \[ (S(1)f)(0)=\lim_{n\to\infty}\frac{1}{n}\bar{\E}
 	\left[nf\left(\frac{1}{\sqrt{n}}\sum_{i=1}^n \xi_i\right)\right]
 	\quad\mbox{for all } f\in\Cb. \]
\end{theorem}

It turns out that, instead of the martingale constraint, it is sufficient to require that the 
couplings used to define the transport cost satisfy the condition
\begin{equation} \label{eq:zero}
 \int_{\R^d\times\R^d} x^T  a(y-x)\,\pi(\d x,\d y)=0 \mbox{ for all } a\in\S_d, 
\end{equation}
where $\S_d$ denotes the set of all symmetric $d\times d$-matrices. We denote 
the set of all couplings $\pi\in\Pi(\mu,\nu)$ satisfying equation~\eqref{eq:zero} by
$\Pi_0(\mu,\nu)$ and define the corresponding transport cost by
\[ C_0(\mu,\nu):=\left(\inf_{\pi\in\Pi_0(\mu,\nu)}
	\int_{\R^d\times\R^d}|x-y|^p\,\pi(\d x,\d y)\right)^\frac{1}{p}, \]
where $\inf\emptyset=\infty$. The convex expectation is then defined by 
\begin{equation} \label{eq:Wasser3.E}
 \E\colon {\rm C}_{\tilde{\kappa}}\to\R,\; f\mapsto\sup_{\nu\in\p_p^0}
 \left(\int_{\R^d}f(x)\,\nu(\d x)-\phi(C_0(\mu,\nu))\right).
\end{equation}
Let $I(0):=\id_{\Cb}$ and, for every $t>0$, $f\in\Cb$ and $x\in\R^d$, 
\[ (I(t)f)(x):=t\E\left[\frac{1}{t}f(x+\sqrt{t}\xi)\right], 
	\quad\mbox{where } \xi:=\id_{\R^d}. \]
Again, one can show that
\[ (I(t)f)(x)=\sup_{\nu\in\p_p^0}
	\left(\int_{\R^d} f(x+y)\,\nu(\d y)-t\phi\left(\frac{C_0(\mu_t,\nu)}{\sqrt{t}}\right)\right), \]
where $\mu_t:=\mu\circ(\sqrt{t}\xi)^{-1}$. This setting leads to the following result.

\begin{theorem} \label{thm:Wasser3}
 The functional $\E$ defined by equation~\eqref{eq:Wasser3.E} is a convex expectation 
 which is continuous from above and satisfies $\E[a\xi]=0$ for all $a\in\R^d$. Hence, 
 there exists a strongly continuous convex monotone semigroup $(S(t))_{t\geq 0}$ on $\Cb$ with
 \[ S(t)f=\lim_{n\to\infty}I\big(\tfrac{t}{n}\big)^n f\in\Cb
 	\quad\mbox{for all } t\geq 0 \mbox{ and } f\in\Cb \]
 which has the properties~(i)-(v) from Theorem~\ref{thm:cher} with $\omega=L=0$.
 For every $f\in\Cb^2$ and $x\in\R^d$, it holds that $f\in D(A)$ and 
 \[ (Af)(x)=\sup_{\lambda\in\R^d}
 	\left(\frac{1}{2}\tr(\lambda\lambda^T D^2f(x))-\phi(|\lambda|)\right)
 	+\frac{1}{2}\tr(\Sigma D^2f(x)), \]
 where $\Sigma:=\int_{\R^d} yy^T\,\mu(\d y)\in\S_d^+$. Moreover, for every convex 
 expectation space $(\Omega,\H,\bar{\E})$ and iid sequence $(\xi_n)_{n\in\N}\subset\H^d$ 
 with $\bar{\E}[f(\xi_n)]=\E[f]$ for all $f\in\Cb$, 
 \[ (S(1)f)(0)=\lim_{n\to\infty}\frac{1}{n}\bar{\E}
 	\left[nf\left(\frac{1}{\sqrt{n}}\sum_{i=1}^n \xi_i\right)\right]
 	\quad\mbox{for all } f\in\Cb. \]
\end{theorem}
\begin{proof} 
 First, we show that $\E$ is a convex expectation which is continuous from above
 and satisfies $\E[a\xi]=0$ for all $a\in\R^d$. Let $f\in {\rm C}_{\tilde{\kappa}}$ 
 and choose $c\geq 0$ with $|f(x)|\leq c(1+|x|^2)$ for all $x\in\R^d$. For every 
 $\nu\in\p_p^0$, it follows from equation~\eqref{eq:zero} that
 \begin{align*}
  \int_{\R^d}|y|^2\,\nu(\d y)-\int_{\R^d}|x|^2\,\mu(\d x)
  &=\int_{\R^d\times\R^d}|x-y|^2\,\pi(\d x,\d y)+2\int_{\R^d\times\R^d}x(y-x)\,\pi(\d x,\d y) \\
  &=\int_{\R^d\times\R^d}|x-y|^2\,\pi(\d x,\d y)
 \end{align*}
 for all $\pi\in\Pi_0(\mu,\nu)$ and therefore H\"older's inequality yields
 \[ \int_{\R^d}|y|^2\,\nu(\d y)-\int_{\R^d}|x|^2\,\mu(\d x)\leq C_0(\mu,\nu)^2. \]
 Hence, we can use $\lim_{c\to\infty}\nicefrac{\phi(c)}{c^2}=\infty$ to obtain 
 \begin{align*}
  \E[|f|]
  &\leq\sup_{\nu\in\p_p^0}\left(\int_{\R^d}c(1+|x|^2)\, \nu(\d x)-\phi(C_0(\mu,\nu))\right) \\
  &=\sup_{\nu\in\p_p^0}\left(\int_{\R^d}c(1+|x|^2)\, \mu(\d x)+\int_{\R^d}c|x|^2\, \nu(\d x)
  -\int_{\R^d}c|x|^2\, \mu(\d x)-\phi(C_0(\mu,\nu))\right) \\
  &\leq\int_{\R^d}c(1+|x|^2)\,\mu(\d x)
  +\sup_{\nu\in\p_p^0}\left(C_0(\mu,\nu)^2-\phi(C_0(\mu,\nu))\right) \\
  &\leq\int_{\R^d}c(1+|x|^2)\,\mu(\d x)
  +\sup_{\nu\in\p_p^0}\left(C_0(\mu,\nu)^2-\phi(C_0(\mu,\nu))\right)<\infty.
 \end{align*}  
 It follows from $\phi(0)=0$ that $\E[c]=c$ for all $c\in\R$. Moreover, the functional 
 $\E$ is clearly convex and monotone. For every $a\in\R^d$, we use the fact that 
 $\int_{\R^d}x\,\nu(\d x)=0$ for all $\nu\in\p_p^0$ and the condition $\phi(0)=0$ to
 conclude $\E[a\xi]=0$. Let $(f_n)_{n\in\N}\subset {\rm C}_{\tilde{\kappa}}$ be a sequence with 
 $f_n\downarrow 0$ and choose $c\geq 0$ with $f_1(x)\leq c(1+|x|^2)$ for all 
 $x\in\R^d$. Since 
 \[ \int_{\R^d} f_n(x)\,\nu(\d x)-\phi(C_0(\mu,\nu))
 	\leq\int_{\R^d}c(1+|x|^2)\,\mu(\d x)+C_0(\mu,\nu)^2-\phi(C_0(\mu,\nu)) \]
 for all $\nu\in\p_p^0$ and $\lim_{c\to\infty}\nicefrac{\phi(c)}{c^2}=\infty$, there exists 
 $R\geq 0$ with
 \[ \E[f_n]=\sup_{\nu\in M}\left(\int_{\R^d} f_n(x)\,\nu(\d x)-\phi(C_0(\mu,\nu))\right)
 	\leq\sup_{\nu\in M}\int_{\R^d} f_n(x)\,\nu(\d x), \]
 where $M:=\{\nu\in\p_p^0\colon C_0(\mu,\nu)\leq R\}$. For every $\epsilon>0$, we use 
 $\W_p(\mu,\nu)\leq C_0(\mu,\nu)$ and Lemma~\ref{lem:tightW} to choose $r\geq 0$ with 
 \[ \sup_{\nu\in M}\int_{B(r)^c}c(1+|x|^2)\,\nu(\d x)\leq\frac{\epsilon}{2}. \]
 Moreover, we can use Dini's theorem to choose $n_0\in\N$ with
 \[ \int_{\R^d} f_n(x)\,\nu(\d x)
 	\leq\int_{B(r)}f_n(x)\,\nu(\d x)+\int_{B(r)^c}c(1+|x|^2)\,\nu(\d x)\leq\epsilon \]
 for all $n\geq n_0$ and $\nu\in M$. We obtain $\E[f_n]\downarrow 0$ as $n\to\infty$. 
 Now, Theorem~\ref{thm:CLT} yields the existence of the semigroup $(S(t))_{t\geq 0}$.
 Furthermore, for every $f\in\Cb^2$ and $x\in\R^d$,
 \[ (Af)(x)=\E\left[\frac{1}{2}\xi^T D^2f(x)\xi\right]. \]
 
 Second, for every $f\in\Cb$ and $x\in\R^d$, we show that 
 \[ (Af)(x)=\sup_{\lambda\in\R^d}
 	\left(\frac{1}{2}\tr(\lambda\lambda^T D^2f(x))-\phi(|\lambda|)\right)
 	+\frac{1}{2}\tr(\Sigma D^2f(x)). \]
 For every $\nu\in\p_p^0$ with $c:=C_0(\mu,\nu)<\infty$ and $\pi\in\Pi_0(\mu,\nu)$, 
 inequality~\eqref{eq:zero} yields
 \begin{align*}
  &\frac{1}{2}\int_{\R^d}z^T D^2f(x)z\,\nu(\d z)-\phi(C_0(\mu,\nu)) \\
  &=\frac{1}{2}\int_{\R^d\times\R^d}z^T D^2f(x)z-y^T D^2f(x)y\,\pi(\d y,\d z)
  +\frac{1}{2}\int_{\R^d}y^T D^2f(x)y\,\mu(\d y)-\phi(c) \\
  &=\frac{1}{2}\int_{\R^d\times\R^d}(z-y)^T D^2f(x)(z-y)\,\pi(\d y,\d z)
  +\int_{\R^d\times\R^d}y^T D^2f(x)(z-y)\,\pi(\d y,\d z) \\
  &\quad\; +\frac{1}{2}\tr(\Sigma D^2f(x))-\phi(c) \\
  &\leq\frac{1}{2}\left(\int_{\R^d\times\R^d}|y-z|^2\,\pi(\d y,\d z)\right)
  \sup_{|\lambda|=1}\lambda^T D^2f(x)\lambda+\frac{1}{2}\tr(\Sigma D^2f(x))-\phi(c).
 \end{align*}
 We take the supremum over $\pi\in\Pi_0(\mu,\nu)$ to obtain
 \begin{align*}
  (Af)(x) &\leq\sup_{c\geq 0}\left(\frac{1}{2}c^2\sup_{|\lambda|=1}
  \tr(\lambda\lambda^T D^2f(x))-\phi(c)\right)+\frac{1}{2}\tr(\Sigma D^2f(x)) \\
  &=\sup_{\lambda\in\R^d}\left(\frac{1}{2}\tr(\lambda\lambda^T D^2f(x))-\phi(|\lambda|)\right)
  +\frac{1}{2}\tr(\Sigma D^2f(x)).
 \end{align*}
 In order to show the reverse inequality, let $\lambda\in\R^d$ and 
 $\nu:=\mu*\big(\frac{1}{2}\delta_\lambda+\frac{1}{2}\delta_{-\lambda}\big)\in\p_p^0$. 
 Furthermore, let $X,Z$ be independent random variables on a probability space 
 $(\Omega,\F,\P)$ with $\P\circ X^{-1}=\mu$ and 
 $\P\circ Z^{-1}=\frac{1}{2}\delta_\lambda+\frac{1}{2}\delta_{-\lambda}$. Define
 $Y:=X+Z$ and $\pi:=\P\circ (X,Y)^{-1}$. Clearly, it holds that $\pi\in\Pi_0(\mu,\nu)$ and
 \[ C_0(\mu,\nu)\leq\e[|X-Y|^p]^\frac{1}{p}=\e[|Z|^p]^\frac{1}{p}=|\lambda|. \]
 Hence, we can use the independence of $X$ and $Z$ to estimate
 \begin{align*}
  (Af)(x) &\geq\frac{1}{2}\int_{\R^d} y^T D^2f(x)y\,\nu(\d y)-\phi(C_0(\mu,\nu)) \\
  &\geq\frac{1}{2}\e[(X+Z)^T D^2f(x)(X+Z)]-\phi(|\lambda|) \\
  &=\frac{1}{2}\tr(\lambda\lambda^T D^2f(x))-\phi(|\lambda|)+\frac{1}{2}\tr(\Sigma D^2f(x)).
 \end{align*}
 Taking the the supremum over $\lambda\in\R^d$ yields
  \[ (Af)(x)\geq\sup_{\lambda\in\R^d}\left(\frac{1}{2}\tr(\lambda\lambda^T D^2f(x))-\phi(|\lambda|)\right)
 	+\frac{1}{2}\tr(\Sigma D^2f(x)). \qedhere \]
\end{proof}

We want to remark that Theorem~\ref{thm:Wasser2} can be obtained from the previous proof. 
Indeed, it follows from $\Pi_M(\mu,\nu)\subset\Pi_0(\mu,\nu)$ that 
$C_M(\mu,\nu)\geq C_0(\mu,\nu)$. Furthermore, the measure $\nu$ and coupling $\pi$
that we have chosen to show
\[ (Af)(x)\geq\sup_{\lambda\in\R^d}\left(\frac{1}{2}\tr(\lambda\lambda^T D^2f(x))-\phi(|\lambda|)\right)
 	+\frac{1}{2}\tr(\Sigma D^2f(x)) \]
satisfy $\pi\in\Pi_M(\mu,\nu)$ and $C_M(\mu,\nu)\leq |\lambda|$. Corollary~\ref{cor:IJ} implies
that the semigroups from Theorem~\ref{thm:Wasser2} and Theorem~\ref{thm:Wasser3} coincide. 
We conclude this subsection by showing that the semigroup from Theorem~\ref{thm:Wasser3} 
corresponding to a supremum over an infinite dimensional set of arbitrary distributions can be 
approximated by random walks. To do so, we define the convex expectation
\[ \tilde{\E}\colon {\rm C}_{\tilde{\kappa}}\to\R,\; f\mapsto\sup_{\lambda\in\R^d}
	\left(\frac{1}{2}\int_{\R^d} f(x+\lambda)+f(x-\lambda)\,\mu(\d x)-\phi(|\lambda|)\right). \]
Moreover, we define $J(0):=\id_{\Cb}$ and, for every $t>0$, $f\in\Cb$ and $x\in\R^d$, 
\[ (J(t)f)(x):=t\tilde{\E}\left[\frac{1}{t}f(x+\sqrt{t}\xi)\right]. \]

\begin{corollary}
 Denoting by $(S(t))_{t\geq 0}$ the semigroup from Theorem~\ref{thm:Wasser3},
 we have 
 \[ S(t)f=\lim_{n\to\infty}J\big(\tfrac{t}{n}\big)^n f
 	\quad\mbox{for all } t\geq 0 \mbox{ and } f\in\Cb. \]
\end{corollary} 
\begin{proof}
 As seen during the proof of Theorem~\ref{thm:Wasser3}, it holds that 
 \begin{align*}
  \tilde{\E}[f] 
  &=\sup_{\lambda\in\R^d}\left(\int_{\R^d} f(x)\,\nu_\lambda(\d x)-\phi(|\lambda|)\right) \\
  &\leq\sup_{\lambda\in\R^d}\left(\int_{\R^d} f(x)\,\nu_\lambda(\d x)-\phi(C_0(\mu,\nu_\lambda))\right)
  \leq\E[f],
 \end{align*}
 where $\nu_\lambda:=\mu*\big(\frac{1}{2}\delta_{-\lambda}+\frac{1}{2}\delta_\lambda\big)$. 
 Hence, the functional $\tilde{\E}\colon {\rm C}_{\tilde{\kappa}}\to\R$ is a well-defined convex expectation which is 
 continuous from above and satisfies $\tilde{\E}[a\xi]=0$ for all $a\in\R^d$. By Theorem~\ref{thm:CLT}, 
 there exists a semigroup $(T(t))_{t\geq 0}$ on ${\rm C}_{\tilde{\kappa}}$ with 
 \[ T(t)f=\lim_{n\to\infty}J(\pi_n^t)f 
 	\quad\mbox{for all } t\geq 0 \mbox{ and } f\in\Cb \]
 and generator $(Bf)(x)=\tilde{\E}[\frac{1}{2}\xi^T D^2f(x)\xi]$ for all $f\in\Cb^2$ and 
 $x\in\R^d$. For every $f\in\Cb^2$ and $x\in\R^d$, it follows from Theorem~\ref{thm:Wasser3} 
 and a straightforward computation that
 \[ (Af)(x)=\sup_{\lambda\in\R^d}\left(\frac{1}{2}\tr(\lambda\lambda^T D^2f(x))-\phi(|\lambda|)\right)
 	+\tr(\Sigma D^2f(x))=(Bf)(x). \]
 Hence, Corollary~\ref{cor:IJ} implies $S(t)f=T(t)f$ for all $t\geq 0$ and $f\in\Cb$.
\end{proof}

\section{Conclusion}

Based on the abstract results in~\cite{BDKN}, we have derived explicit and verifiable
conditions under which strongly continuous convex monotone semigroups are uniquely determined by
their generators evaluated at smooth functions and can be constructed as the limit of
iterative approximation schemes. Hence, we are able to extend previous works in this 
direction and can provide a semigroup approach to HJB equations which is independent from 
the theory of viscosity solutions. In the application to limit theorems we are, to the best
of our knowledge, the first ones to work with convex rather than sublinear expectations. 
While this does not affect the proofs, apart from sometimes longer estimates, it opens
the door to new applications in this context such as large deviations theory. A future 
potential application, where it is crucial to replace sublinearity by convexity, are 
transitions from discrete to continuous models in the context of mathematical finance.
In the sequel, we briefly discuss some further questions. Extending the results in
Section~\ref{sec:LLN} and Section~\ref{sec:CLT} to $\alpha$-stable distributions,
as it has been achieved in~\cite{BM,HJLP} for sublinear expectations, should be 
attainable with the tools developed in this paper. Furthermore, the corresponding 
convergence rates for LLN and CLT type results have been established in~\cite{BJKL} 
and are consistent with the ones in~\cite{HJL21, HL20, Krylov, Song}. While the results 
in this article are the same whether we consider sublinear or convex expectations and 
semigroups, the convergence rates in~\cite{BJKL} strongly depend on how nonlinear the 
expectations and semigroups are. 
Finally, it would be desirable to have an in-depth comparison between the established
viscosity approach to HJB equations and the novel semigroup approach which provides a 
comparison principle and stability results that are very similar to the well-known 
properties of viscosity solutions, see~\cite{BKN}.

\appendix

\section{A basic convexity estimate}

\begin{lemma} \label{lem:lambda}
 Let $X$ be a vector space and $\phi\colon X\to\R$ be a convex functional. Then, 
 \[ \phi(x)-\phi(y)\leq\lambda\left(\phi\left(\frac{x-y}{\lambda}+y\right)-\phi(y)\right)
 	\quad\mbox{for all } x,y\in X \mbox{ and } \lambda\in (0,1]. \]
\end{lemma}
\begin{proof}
 We use the convexity to estimate
 \begin{align*}
  \phi(x)-\phi(y)
  &=\phi\left(\lambda\left(\frac{x-y}{\lambda}+y\right)+(1-\lambda)y\right)-\phi(y) \\
  &\leq\lambda\phi\left(\frac{x-y}{\lambda}+y\right)+(1-\lambda)\phi(y)-\phi(y) \\
  &=\lambda\left(\phi\left(\frac{x-y}{\lambda}+y\right)-\phi(y)\right)
 \end{align*}
 for all $x,y\in X$ and $\lambda\in (0,1]$. 
\end{proof}

\section{Convex expectation spaces} \label{app:E}

Nonlinear expectations were introduced by Peng, see~\cite{Peng19} for a detailed 
discussion, and are a closely related to several other concepts. For instance, sublinear 
expectations are called upper expectation in robust statistics~\cite{Huber}, upper coherent 
prevision in the theory of imprecise probabilities~\cite{Walley} and coherent risk 
measure~\cite{ADEH} in mathematical finance. Moreover, a convex expectation coincides 
(up to the sign) with the notion of a convex risk measure~\cite{FS02,FG}. In the sequel,
we mainly follow~\cite[Chapter~1]{Peng19} up to some direct transfers from the sublinear 
to the convex case. All inequalities and the monotone convergence in the following
definition are understood pointwise.

\begin{definition}
 Let $\Omega$ be a set and $\H$ a linear space of functions $X\colon\Omega\to\R$
 with $c\in\H$ and $|X|\in\H$ for all $c\in\R$ and $X\in\H$.\footnote{We identify $c$ 
 with the constant function $c\one_\Omega$.} A convex expectation is a functional
 $\E\colon\H\to\R$ with
 \begin{itemize}
 \item $\E[c]=c$ for all $c\in\R$, 
 \item $\E[X]\leq\E[Y]$ for all $X,Y\in\H$ with $X\leq Y$, 
 \item $\E[\lambda X+(1-\lambda)Y]\leq\lambda\E[X]+(1-\lambda)\E[Y]$ for all 
  $X,Y\in\H$ and $\lambda\in[0,1]$. 
 \end{itemize}
 The triplet $(\Omega,\H,\E)$ is called a convex expectation space. Furthermore, 
 we say that $\E$ is continuous from above if $\E[X_n]\downarrow 0$ for all 
 $(X_n)_{n\in\N}\subset\H$ with $X_n\downarrow 0$. If $\E[\lambda X]=\lambda\E[X]$ 
 for all $\lambda\geq 0$ and $X\in\H$ we say that $\E$ is sublinear and call
 $(\Omega,\H,\E)$ a sublinear expectation space.
\end{definition}

If $\E$ is continuous from above it follows from the convexity that $\E[X_n]\downarrow\E[X]$
for all $(X_n)_{n\in\N}\subset\H$ and $X\in\H$ with $X_n\downarrow X$. We collect some 
elementary properties of convex expectations.

\begin{lemma} \label{lem:E}
 For a convex expectation space $(\Omega,\H,\E)$ the following statements hold:
 \begin{itemize}
  \item[(i)] $\E[X+c]=\E[X]+c$ for all $X\in\H$ and $c\in\R$.
  \item[(ii)] $|\E[X]-\E[Y]|\leq\|X-Y\|_\infty$ for all bounded $X,Y\in\H$. 
  \item[(iii)] $\E[\lambda X]\leq\lambda\E[X]$ for all $\lambda\in [0,1]$ and $X\in\H$.
  \item[(iv)] $-\E[-X]\leq\E[X]$ for all $X\in\H$. 
  \item[(v)] $|\E[X]|\leq\E[|X|]$ for all $X\in\H$.
  \item[(vi)] Let $X\in\H$ with $\E[aX]=0$ for all $a\in\R$. Then, it holds that 
   \[ \E[X+Y]=\E[Y] \quad\mbox{for all } Y\in\H. \]
 \end{itemize}
\end{lemma}
\begin{proof}
 \begin{itemize}
  \item[(i)] For every $X\in\H$, $c\in\R$ and $\lambda\in (0,1)$, we use that $\E$ is
   convex and preserves constants to estimate
   \[ \E[X+c] 
   	\leq\lambda\E\left[\frac{X}{\lambda}\right]+(1-\lambda)\E\left[\frac{c}{1-\lambda}\right]
    =\lambda\E\left[\frac{X}{\lambda}\right]+c \]
   and 
   \[ \E[X]=\E[X+c-c]
    \leq\lambda\E\left[\frac{X+c}{\lambda}\right]+(1-\lambda)\E\left[-\frac{c}{1-\lambda}\right]
    =\lambda\E\left[\frac{X+c}{\lambda}\right]-c. \]
   Since the real-valued mapping $\lambda\mapsto\E[\lambda X]$ is convex and therefore 
   continuous, we obtain in the limit $\lambda\to 1$ that $\E[X+c]=\E[X]+c$.  
  \item[(ii)] It follows from the monotonicity and part~(i) that 
   \[ \E[X]\leq\E[Y+\|X-Y\|_\infty]=\E[Y]+\|X-Y\|_\infty. \]
   Changing the roles of $X$ and $Y$ yields the claim.
  \item[(iii)] For every $\lambda\in [0,1]$ and $X\in\H$, 
   \[ \E[\lambda X]=\E[\lambda X+(1-\lambda)0]\leq\lambda\E[X]+(1-\lambda)\E[0]
   	=\lambda\E[X].\]
  \item[(iv)] It holds that $0=\E[0]=\E[\frac{1}{2}X+\frac{1}{2}(-X)]\leq\frac{1}{2}\E[X]+\frac{1}{2}\E[-X]$. 
  \item[(v)] We use the monotonicity and part~(iv) to estimate
   \[ -\E[|X|]\leq -\E[-X]\leq\E[X]\leq\E[|X|].\]
  \item[(vi)] For every $\lambda\in (0,1]$, it follows from Lemma~\ref{lem:lambda} that
   \begin{align*}
    \E[X+Y]-\E[Y] 
   	&\leq\lambda\E\left[\frac{X}{\lambda}+Y\right]-\lambda\E[Y] \\
   	&\leq\frac{\lambda}{2}\E\left[\frac{2X}{\lambda}\right]+\frac{\lambda}{2}\E[2Y]
   	-\lambda\E[Y] \\
   &=\frac{\lambda}{2}\E[2Y]-\lambda\E[Y] \to 0 \quad\mbox{as } \lambda\downarrow 0.
   \end{align*}
   Similarly one can show that
   \begin{align*}
     \E[X+Y]-\E[Y]
     &\geq -\frac{\lambda}{2}\E\left[-\frac{2X}{\lambda}\right]-\frac{\lambda}{2}\E[2(X+Y)]
     +\lambda\E[X+Y] \\
    &=\frac{\lambda}{2}\E[2(X+Y)]+\lambda\E[X+Y]\to 0 \quad\mbox{as } \lambda\downarrow 0. \qedhere
   \end{align*}
 \end{itemize}
\end{proof}

\begin{definition}
 Let $(\Omega,\H,\E)$ be a convex expectation space with $f(X_1,\ldots,X_n)\in\H$ 
 for all $n\in\N$, $f\in\Cb(\R^n)$ and $X\in\H^n$. Let $m,n\in\N$, $X\in\H^m$ and
 $Y\in\H^n$. 
 \begin{itemize}
  \item[(i)] The distribution of $X$ is given by the functional 
   \[ F_X\colon\Cb(\R^m)\to\R,\; f\mapsto \E[f(X)]. \]
  \item[(ii)] We say that $X$ and $Y$ are identically distributed if $m=n$ and
   \[ \E[f(X)]=\E[f(Y)] \quad\mbox{for all } f\in\Lipb(\R^m). \]
  \item[(iii)] We say that $Y$ is independent of $X$ if
   \[ \E[f(X,Y)]=\E\left[\E[f(x,Y)]\big|_{x=X}\right]
   	\quad\mbox{for all } f\in\Lipb(\R^m\times\R^n). \]
 \end{itemize}
\end{definition}

In general, the statements ``Y is independent of X'' and ``X is independent of Y''
are not equivalent, see~\cite[Example~1.3.15]{Peng19} for a counterexample.

\begin{definition}
 Let $(\Omega,\H,\E)$ be a convex expectation space and $(X_n)_{n\in\N}\subset\H^d$
 be a sequence of random vectors for some $d\in\N$. 
 \begin{itemize}
  \item[(i)] We say that $(X_n)_{n\in\N}$ is independent and identically distributed (iid)
   if $X_m$ and $X_n$ have the same distribution and $X_{n+1}$ is independent of 
   $(X_1,\ldots,X_n)$ for all $m,n\in\N$. 
  \item[(ii)] We say that $(X_n)_{n\in\N}$ converges in distribution if the sequence
   $(\E[f(X_n)])_{n\in\N}$ converges for all $f\in\Lipb(\R^d)$. 
 \end{itemize}
\end{definition}

Let $\Omega_\infty:=(\R^d)^\N$. For every $n\in\N$, we define
$\pi_n\colon\Omega_\infty\to (\R^d)^n,\; x\mapsto (x_1,\ldots,x_n)$ and  
$\H_n:=\{f\circ\pi_n\colon f\in\Lipb((\R^d)^n)\}$. Furthermore, let $\H_\infty:=\bigcup_{n\in\N}\H_n$. 
For a convex expectation $\E$ on $(\R^d,\Lipb(\R^d))$, we define recursively a sequence
of convex expectations $\E_n\colon\Lipb((\R^d)^n)\to\R$ by $\E_1:=\E$ and
\[ \E_{n+1}[f]:=\E_n\big[\E[f(x,Y)]\big|_{x=X}\big], \]
where $X:=\id_{(\R^d)^n}$ and $Y:=\id_{\R^d}$. Note that the function
$x\mapsto\E[f(x,Y)]$ is Lipschitz continuous, because Lemma~\ref{lem:E}(ii) implies
\[ |\E[f(x,\cdot)]-\E[f(y,\cdot)]|\leq\|f(x,\cdot)-f(y,\cdot)\|_\infty\leq r|x-y| \]
for all $r\geq 0$, $f\in\Lipb(r)$ and $x,y\in\R^d$. Furthermore, let
\[ \E_\infty\colon\H_\infty\to\R,\; f\circ\pi_n\mapsto\E_n[f]. \]
The next result is a direct transfer of~\cite[Proposition~1.3.17]{Peng19} from the sublinear
to the convex case.

\begin{lemma}
 Let $\E$ be a convex expectation on $(\R^d,\Lipb(\R^d))$ and define 
 $(\Omega_\infty,\H_\infty,\E_\infty)$ as above. Furthermore, let
 \[ \xi_n\colon\Omega_\infty\to\R, \; x\mapsto x_n \quad\mbox{for all } n\in\N. \]
 Then, the sequence $(\xi_n)_{n\in\N}$ is iid and satisfies $\E_\infty[f(\xi_n)]=\E[f]$ 
 for all $f\in\Lipb(\R^d)$. 
\end{lemma}

The next result is a direct application of~\cite[Theorem~4.6]{DKN2018}. Let
$\B^\N:=\bigotimes_{n\in\N}\B(\R)$ be the product-$\sigma$-algebra, where $\B(\R)$ 
denotes the Borel-$\sigma$-algebra, and define $\L^\infty(\Omega_\infty)$ as the 
space of all bounded $\B^\N$-measurable functions $f\colon\Omega_\infty\to\R$.

\begin{theorem}[Kolmogorov] \label{thm:kol}
 Let $\E$ be a convex expectation on $(\R^d,\Cb(\R^d))$ which is continuous from above
 and define $(\Omega_\infty,\H_\infty,\E_\infty)$ as before. Then, there exists a unique 
 convex expectation $\bar{\E}_\infty\colon\L^\infty(\Omega_\infty)\to\R$ with 
 $\bar{\E}_\infty[f]=\E_\infty[f]$ for all $f\in\H_\infty$ which is continuous from below
 on $\L^\infty(\Omega_\infty)$ and continuous from above on
 \[ (H_\infty)_\delta:=\{f\in\L^\infty(\Omega_\infty)\colon\mbox{there exists } 
 	(f_n)_{n\in\N}\subset\H_\infty\mbox{ with } f_n\downarrow f\}. \]
 Moreover, the sequence $(\xi_n)_{n\in\N}$ defined by
 \[ \xi_n\colon\Omega_\infty\to\R, \; x\mapsto x_n \quad\mbox{for all } n\in\N \]
 is iid and satisfies $\E_\infty[f(\xi_n)]=\E[f]$ for all $f\in\Cb(\R^d)$. 
\end{theorem}

While the property that a random vector has mean zero w.r.t.  a convex expectation 
might seem quite restrictive, this can always be achieved by a simple transformation
of the convex expectation if the random vector has non-negative mean. If the convex
expectation is defined as a supremum over a set of probability measures, the transformed
expectation will be given by a supremum over a smaller set of measures which satisfy
an additional constraint, see Subsection~\ref{sec:Wasser2}.

\begin{lemma} \label{lem:mean}
 Let $(\Omega,\H, \E)$ be a convex expectation space and $\xi\in\H^d$ with
 $\E[a\xi]\geq 0$ for all $a\in\R^d$. Then,
 \[ \tilde{\E}\colon\H\to\R,\; X\mapsto\inf_{a\in\R^d}\E[X+a\xi] \]
 is a convex expectation with $\tilde{\E}[a\xi]=0$ for all $a\in\R^d$. If $\E$ is continuous 
 from above, the same holds for $\tilde{\E}$. Moreover, in the case $\E[a\xi]=0$ for all 
 $a\in\R^d$, it holds that $\E=\tilde{\E}$. 
\end{lemma}
\begin{proof}
 Clearly, the functional $\tilde{\E}$ is monotone and satisfies $\tilde{\E}[c]=c$ for all 
 $c\in\R$. For every $X,Y\in\H$, $\lambda\in [0,1]$ and $a,b\in\R^d$,
 \begin{align*}
  \tilde{\E}[\lambda X+(1-\lambda)Y]
  &\leq\E[\lambda X+(1-\lambda)Y+(\lambda a+(1-\lambda)b)\xi] \\
  &\leq\lambda\E[X+a\xi]+(1-\lambda)\E[Y+b\xi].
 \end{align*}
 Taking the infimum over $a,b\in\R^d$ yields 
 \[ \tilde{\E}[\lambda X+(1-\lambda)Y]\leq\lambda\tilde{\E}[X]+(1-\lambda)\tilde{\E}[Y]. \]
 Since $\E[a\xi+b\xi]=\E[(a+b)\xi]\geq 0$ for all $a,b\in\R^d$ with equality for $b=-a$,
 we obtain $\tilde{\E}[a\xi]=0$ for all $a\in\R^d$. Assume that $\E$ is continuous from 
 above and let $(X_n)_{n\in\N}\subset\H$ with $X_n\downarrow 0$. Then, 
 \[ \inf_{n\in\N}\tilde{\E}[X_n]=\inf_{n\in\N}\inf_{a\in\R^d}\E[X_n+a\xi]
 	=\inf_{a\in\R^d}\inf_{n\in\N}\E[X_n+a\xi]=\inf_{a\in\R^d}\E[a\xi]=0 \]
 which shows that $\tilde{\E}$ is continuous from above. If $\E[a\xi]=0$ for all $a\in\R^d$,
 it follows from Lemma~\ref{lem:E}(vi) that $\tilde{\E}[X]=\inf_{a\in\R^d}\E[X+a\xi]=\E[X]$
 for all $X\in\H$. 
\end{proof}

\section{Tightness of Wasserstein balls}

Let $p,q\in [1,\infty)$ with $p>q$ and denote by $\p_p$ the $p$-Wasserstein space endowed
with the $p$-Wasserstein distance $\W_p$.

\begin{lemma} \label{lem:tightW}
 Let $\mu\in\p_p$, $R\geq 0$ and $M:=\{\nu\in\p_p\colon\W_p(\mu,\nu)\leq R\}$. 
 Then, for every $\epsilon>0$, there exists $r\geq 0$ with 
 \[ \sup_{\nu\in M}\int_{B(r)^c}|x|^q\,\nu(\d x)\leq\epsilon. \]
\end{lemma}
\begin{proof}
 Let $\epsilon>0$ and choose $r_1\geq 0$ with 
 \begin{equation} \label{eq:tightW}
  2^{q-1}\int_{B(r_1)^c}|x|^q\,\mu(\d x)\leq\frac{\epsilon}{4} \quad\mbox{and}\quad 
  2^{q-1}R^\frac{q}{p}\mu\big(B(r_1)^c\big)^\frac{p-q}{p}\leq\frac{\epsilon}{4}.
 \end{equation}
 Let $\nu\in M$ and choose an optimal coupling $\pi\in\Pi(\mu,\nu)$, i.e., 
 \[ \W_p(\mu,\nu)=\left(\int_{\R^d\times\R^d}|x-y|^p\,\pi(\d x, \d y)\right)^\frac{1}{p}. \]
 It follows from H\"older's inequality and inequality~\eqref{eq:tightW} that 
 \begin{align}
  &\int_{B(r_1)^c\times B(r_1)^c}|x-y|^q\,\pi(\d x, \d y) \nonumber \\
  &\leq\left(\int_{\R^d\times\R^d}|x-y|^p\,\pi(\d x, \d y)\right)^\frac{q}{p}
  \pi\big(B(r_1)^c\times B(r_1)^c)^\frac{p-q}{p} \nonumber \\
  &\leq R^\frac{q}{p}\pi\big(B(r_1)^c\times\R^d\big)^\frac{p-q}{p} 
  =R^\frac{q}{p}\mu\big(B(r_1)^c\big)^\frac{p-q}{p}\leq\frac{\epsilon}{4}. 
 \end{align}
 Moreover, for every $x\in B(r_1)$, 
 \[ \frac{|y|^q}{|x-y|^p}\leq\frac{|y|^q}{(|y|-|x|)^p}\leq\frac{|y|^q}{(|y|-r_1)^p}\to 0
 	\quad\mbox{as } |y|\to\infty. \]
 Hence, we can choose $r_2\geq r_1$ with 
 \begin{equation} \label{eq:tightW2}
  \frac{|y|^q}{(|x-y|)^p}\leq\frac{\epsilon}{2R^p}
  \quad\mbox{for all } x\in B(r_1) \mbox{ and } y\in B(r_2)^c. 
 \end{equation}
 It follows from inequality~\eqref{eq:tightW}-\eqref{eq:tightW2} that
 \begin{align*}
  &\int_{B(r_2)^c}|y|^q\, \nu(\d y)
  =\int_{B(r_1)\times B(r_2)^c}|y|^q\,\pi(\d x, \d y)
  +\int_{B(r_1)^c\times B(r_2)^c}|y|^q\,\pi(\d x, \d y) \\
  &\leq\int_{B(r_1)\times B(r_2)^c}\frac{|y|^q}{|x-y|^p}\cdot |x-y|^p\,\pi(\d x, \d y) \\
  &\quad\; +2^{q-1}\int_{B(r_1)^c\times B(r_2)^c}|x|^q+|x-y|^q\,\pi(\d x, \d y) \\
  &\leq\frac{\epsilon}{2R^p}\int_{\R^d\times\R^d}|x-y|^p\,\pi(\d x, \d y)
  +2^{q-1}\int_{B(r_1)^c}|x|^q\,\mu(\d x) \\
  &\quad\; +2^{q-1}\int_{B(r_1)^c\times B(r_1)^c}|x-y|^q\,\pi(\d x, \d y) \\
  &\leq\frac{\epsilon}{2R^p}\W_p(\mu,\nu)^p+\frac{\epsilon}{2}\leq\epsilon.  \qedhere
 \end{align*}
\end{proof}

\bibliographystyle{abbrv} 
\bibliography{clt}

\end{document}